\documentclass[11pt]{amsart}
\usepackage{amsmath,amsfonts,amssymb,amscd,mathrsfs,pstricks,pst-plot}

\input xy
\xyoption{all}

\newtheorem{theorem}{Theorem}[section]
\newtheorem{lemma}[theorem]{Lemma}
\newtheorem{corollary}[theorem]{Corollary}
\newtheorem{proposition}[theorem]{Proposition}

\theoremstyle{definition}
\newtheorem{definition}[theorem]{Definition}
\newtheorem{example}[theorem]{Example}
\newtheorem{problem}[theorem]{Problem}
\newtheorem{remark}[theorem]{Remark}

\numberwithin{equation}{section}

\newcommand{\C}{\mathbb{C}}

\newcommand{\N}{\mathbb{N}}

\newcommand{\Z}{\mathbb{Z}}
\renewcommand{\P}{\mathbb{P}}
\newcommand{\R}{\mathbb{R}}

\newcommand\Id{\mathrm{Id}}

\newcommand{\cA}{\mathcal{A}}
\newcommand{\cB}{\mathcal{B}}
\newcommand{\cC}{\mathcal{C}}

\newcommand{\cE}{\mathcal{E}}
\newcommand{\cF}{\mathcal{F}}

\newcommand{\cH}{\mathcal{H}}

\newcommand{\cJ}{\mathcal{J}}

\newcommand{\cL}{\mathcal{L}}

\newcommand{\cO}{\mathcal{O}}
\newcommand{\cS}{\mathcal{S}}

\newcommand{\cT}{\mathcal{T}}
\newcommand{\cU}{\mathcal{U}}

\newcommand\di{\partial}

\newcommand\dibar{\overline\partial}

\def\mgot{\mathfrak{m}}

\newcommand\wt{\widetilde}
\newcommand{\Aut}{\mathop{{\rm Aut}}}
\newcommand\dist{\mathrm{dist}}
\newcommand\Crit{\mathrm{Crit}}

\newcommand\reg{\mathrm{reg}}
\newcommand\sing{\mathrm{sing}}
\newcommand\rmax{\mathrm{rmax}}

\def\di{\partial}
\def\dibar{\overline\partial}

\numberwithin{equation}{section}

\begin{document}
\title{Noncritical holomorphic functions on Stein spaces}
\author[Franc Forstneri\v c]{Franc Forstneri\v c}

\address{Franc Forstneri\v c, Faculty of Mathematics and Physics, University of Ljubljana, and Institute of Mathematics, Physics and Mechanics, Jadranska 19, 1000 Ljubljana, Slovenia}
\email{franc.forstneric@fmf.uni-lj.si}

\thanks{The author was supported by the program P1-0291 and the grant J1-5432 from ARRS, Republic of Slovenia.}

\subjclass[2010]{Primary 32C42, 32E10, 32E30. Secondary 57R70, 58K05.}
\date{\today}

\keywords{Holomorphic functions, critical points, Stein manifolds, Stein spaces, $1$-convex manifolds, stratifications.}

\begin{abstract}
In this paper we prove that every reduced Stein space admits a holomorphic function without critical points.
Furthermore, every closed discrete subset of a reduced Stein space $X$ is the critical locus of a holomorphic function on $X$. 
We also show that for every complex analytic stratification with nonsingular strata on a reduced Stein space there exists 
a holomorphic function whose restriction to every stratum is noncritical. These result provide some information 
on critical loci of holomorphic functions on desingularizations of Stein spaces. In particular, every $1$-convex manifold 
admits a holomorphic function that is noncritical outside  the exceptional variety.
\end{abstract}

\maketitle

% \tableofcontents

%%%%%%%%%%%%%%%%%%%%%%%%%%%%%%%%%%%%%%%%%%%%%%%%%%%%%%%%%%%%%%%%%%%%%%%%%%%
%																				%	
%																				%	
%  INTRODUCTION																	%	
%																				%	
%																				%
%%%%%%%%%%%%%%%%%%%%%%%%%%%%%%%%%%%%%%%%%%%%%%%%%%%%%%%%%%%%%%%%%%%%%%%%%%%
%

\section{Introduction}
\label{sec:intro}
Every Stein manifold $X$ admits a holomorphic function $f\in \cO(X)$ without critical points;
see Gunning and Narasimhan \cite{Gunning-Narasimhan} for the case of open Riemann surfaces and 
\cite{FF:Acta} for the general case. In the algebraic category this fails on any compact Riemann surface of 
genus $g\ge 1$ with a puncture (every algebraic function on such a surface has a critical point as follows  
from the Riemann-Hurwitz theorem), 
but it holds for holomorphic functions of finite order \cite{FO}. Noncritical holomorphic functions are of interest in particular 
since they define nonsingular holomorphic hypersurface foliations; results on this topic can be found in \cite{FF:Acta,FF:submersions}. 

In this paper we prove that, somewhat surprisingly, the same holds on Stein spaces. The following 
is a special case of our main result, Theorem \ref{th:mainbis}.

\begin{theorem}
\label{th:main}
Every reduced Stein space admits a holomorphic function without critical points. 
\end{theorem}

We begin by recalling the relevant notions. All complex spaces are assumed paracompact and reduced. 
For the theory of Stein spaces we refer to Grauert and Remmert \cite{Grauert-Remmert1979}. 

Let $X$ be a complex space. Denote by $\cO_{X,x}$ the ring of germs of holomorphic function at a point $x\in X$ and by 
$\mgot_x$ the maximal ideal of $\cO_{X,x}$, so $\cO_{X,x}/\mgot_x \cong \C$. Given $f\in \cO_{X,x}$ we denote by 
$f-f(x) \in \mgot_x$ the germ obtained by subtracting from $f$ its value $f(x)\in\C$ at $x$.

\begin{definition}
\label{def:critical}
Assume that $x$ is nonisolated point of a complex space $X$. 
\begin{itemize}
\item[\rm (a)] A germ $f\in \cO_{X,x}$ at $x$ is said to be {\em critical} (and $x$ is a {\em critical point} of $f$) if 
$f-f(x) \in \mgot_x^2$ (the square of the maximal ideal $\mgot_x$), 
and is {\em noncritical} if $f-f(x) \in \mgot_x\setminus \mgot_x^2$. 
\item[\rm (b)] A germ $f\in \cO_{X,x}$ is {\em strongly noncritical at $x$} if the germ at $x$ of the restriction $f|_V$ to 
any local irreducible component $V$ of $X$ is noncritical.
\end{itemize}
Any function is considered (strongly) noncritical at an isolated point of $X$. 
\end{definition}

One can characterize these notions by the (non) vanishing of the differential $df_x$ on the {\em Zariski tangent space} $T_x X$. 
Recall that $T_xX$ is isomorphic to $(\mgot_x/\mgot_x^2)^*$, the dual of $\mgot_x/\mgot_x^2$, the latter being the cotangent 
space $T_x^*X$ (cf.\ \cite[p.\ 78]{Fischer} or \cite[p.\ 111]{KK}). The number $\dim_\C T_x X$ is the embedding dimension 
of the germ $X_x$ of $X$ at $x$. The differential $df_x\colon T_x X\to \C$ of $f\in\cO_{X,x}$ is determined by the class 
$f-f(x) \in \mgot_x/\mgot_x^2=T_x^*X$, so $f$ is critical at $x$ if and only if $df_x=0$. If $X_x=\bigcup_{j=1}^k V_j$ is a 
decomposition into local irreducible components, then $f$ is strongly noncritical at $x$ if $df_x\colon T_x V_j\to\C$ is 
nonvanishing for every $j=1,\ldots,k$. 

At a regular point $x\in X_\reg$ these notions coincide with the usual one: $x$ is a critical point of $f$ if and only if in some 
(hence in any) local holomorphic coordinates $z=(z_1,\ldots, z_n)$ on a neighborhood of $x$, with $z(x)=0$ and $n=\dim_x X$, 
we have $\frac{\di f}{\di z_j}(0)=0$ for $j=1,\ldots,n$. Hence the set $\Crit(f)$ of all critical points of a holomorphic function 
on a {\em complex manifold} $X$ is a closed complex subvariety of $X$. On a Stein manifold $X$ this set is discrete for a 
generic choice of $f\in\cO(X)$. 

The following is our main result.

\begin{theorem} 
\label{th:mainbis} 
On  every reduced Stein space $X$ there exists a holomorphic function which is strongly noncritical at every point. Furthermore, 
given a closed discrete set $P=\{p_1,p_2,\ldots\}$ in $X$, germs $f_k\in\cO_{X,p_k}$ and integers $n_k\in\N$, there exists a 
holomorphic function $F\in\cO(X)$ which is strongly noncritical on $X\setminus P$ and agrees with the germ $f_k$ to order $n_k$ 
at the point $p_k\in P$ (i.e., $F_{p_k}-f_k\in \mgot_{p_k}^{n_k}$) for every $k=1,2,\ldots$.  
\end{theorem}

The hypothesis on the set $P$ in Theorem \ref{th:mainbis} is a natural one since the critical locus a generic holomorphic function 
on a Stein space is discrete (see Corollary \ref{cor:generic}). 

We also prove the following result (cf.\ Theorem \ref{th:stratbis} and Corollary \ref{cor:extension}). Given a closed complex
subvariety $X'$ of a reduced Stein space $X$ and a function $f \in \cO(X')$, there exists  $F\in \cO(X)$ such that $F|_{X'}=f$ 
and $F$ is strongly noncritical on $X\setminus X'$, or it has critical points at a prescribed discrete set 
contained in $X\setminus X'$. 

The proof  of these results for Stein manifolds in \cite{FF:Acta} relies on two main ingredients: 
\begin{itemize}
\item[\rm (i)]  Runge approximation theorem for noncritical holomorphic functions on polynomially convex subset of $\C^n$ 
by entire noncritical functions (cf.\ \cite[Theorem 3.1]{FF:Acta} or \cite[Theorem 8.11.1, p.\ 381]{FF:book}), and 
\item[\rm (ii)]  a splitting lemma for biholomorphic maps close to the identity on a Cartan pair (cf.\ \cite[Theorem 4.1]{FF:Acta} 
or \cite[Theorem 8.7.2)]{FF:book}).
\end{itemize}

These tools do not apply directly at singular points of $X$. In addition, the following two phenomena make the analysis very delicate. 

Firstly, the critical locus of a holomorphic function $f\in\cO(X)$ need not be a closed complex subvariety of $X$ near a singularity. 
A simple example is $X=\{zw=0\}\subset \C^2_{(z,w)}$ and $f(z,w)=z$ with $\Crit(f)= \{(0,w)\colon w\ne 0\}$; 
for another example on an irreducible isolated surface singularity see Example \ref{ex:null} in \S \ref{sec:prel}. 
However, we will show that $\Crit(f|X_\reg)\cup X_\sing$ is a closed complex subvariety of $X$ (cf.\ Lemma \ref{lem:crit}). 

Secondly, the class of noncritical (or strongly noncritical) functions is not stable under small perturbations on 
compact sets which include singular points of $X$; see Example \ref{ex:null2}. 

The key idea used in this paper stems from the following observation:

\smallskip
\noindent {\em (*) If $S\subset X$ is a local complex submanifold of positive dimension at a point $x\in S$, 
and if the restriction of a function $f\in \cO(X)$ to $S$ is noncritical at $x$, then $f$ is noncritical at $x$ (as a function on $X$). 
If such $S$ is contained in every local irreducible component of $X$ at $x$, then $f$ is strongly noncritical at $x$.}
\smallskip

This observation naturally leads one to consider complex analytic stratifications of a Stein space and to construct 
holomorphic functions that are noncritical on every stratum. 

Recall that a (complex analytic) {\em stratification} $\Sigma=\{S_j\}$ of a complex space $X$ is a subdivision of $X$ into the 
union $X=\bigcup_j S_j$ of at most countably many pairwise disjoint connected complex manifolds $S_j$, called the {\em strata} of 
$\Sigma$, such that
\begin{itemize}
\item 
every compact set in $X$ intersects at most finitely many strata, and
\item $bS=\overline S\setminus S$ is a union of lower dimensional strata for every $S\in\Sigma$.
\end{itemize}
Such a pair $(X,\Sigma)$ is called a {\em stratified complex space}. 

Every complex analytic space admits a stratification (cf.\ Whitney \cite{Whitney,Whitney2}). An example is obtained by 
taking $X=X_0\supset X_1\supset \cdots$, where $X_{j+1}=(X_j)_\sing$ for every $j$, and decomposing the smooth differences 
$X_j\setminus X_{j+1}$ into connected components. This chain of subvarieties is stationary on each compact subset of $X$.

\begin{definition}
\label{def:stratified-noncritical}
Let $(X,\Sigma)$ be a stratified complex space. A function $f\in\cO(X)$ is said to be a {\em stratified noncritical holomorphic
 function} on $(X,\Sigma)$, or a {\em $\Sigma$-noncritical function}, if the restriction $f|_{S}$ to any stratum $S\in \Sigma$ 
 of positive dimension is a noncritical function on $S$.
\end{definition}

Clearly the critical locus of a $\Sigma$-noncritical function on $(X,\Sigma)$ is contained in the union $X_0$ of all 
$0$-dimensional strata of $\Sigma$; note that $X_0$ is a discrete subset of $X$. 

%
%
%  STRATIFIED NONCRITICAL FUNCTIONS
%
%
\begin{theorem} 
\label{th:stratified}
On every stratified Stein space $(X,\Sigma)$ there exists a $\Sigma$-noncritical holomorphic function $F\in \cO(X)$.  
Furthermore, $F$ can be chosen to agree to a given order $n_x\in \N$ with a given germ $f_x\in \cO_{X,x}$ 
at any $0$-dimensional stratum $\{x\}\in \Sigma$. 
\end{theorem}

Theorem \ref{th:stratified} is proved in \S\ref{sec:stratified}. Assuming it for the moment we give

\smallskip
\noindent{\em Proof of Theorems \ref{th:main} and \ref{th:mainbis}.} 
We may assume that $X$ has no isolated points. Choose a complex analytic stratification $\Sigma$ of $X$ containing a given 
discrete set $P\subset X$ in the union $X_0=\{p_1,p_2,\ldots\}$ of its zero dimensional strata. For every $i=1,2,\ldots$ let $X_i$
denote the union of all strata of dimension at most $i$ (the {\em $i$-skeleton} of $\Sigma$). Note that $X_i$ is a closed complex
subvariety of $X$ (since the boundary of each stratum is a union of lower dimensional strata), the difference 
$X_{i}\setminus X_{i-1}$ is either empty or a complex manifold of dimension $i$, and 
\begin{equation}
\label{eq:skeleton}
	X_0\subset X_1\subset X_2\subset\cdots \subset \bigcup_{i=0}^\infty X_i=X. 
\end{equation}
Given germs $f_k\in \cO_{X,p_k}$ $(p_k\in X_0)$ and integers $n_k\in \N$, Theorem \ref{th:stratified} furnishes a 
$\Sigma$-noncritical function $F\in\cO(X)$ such that $F_{p_k}-f_{p_k} \in \mgot_{p_k}^{n_k}$ for every $p_k\in X_0$. 
We claim that $F$ is strongly noncritical on $X\setminus X_0$. Indeed, given a point $x\in X\setminus X_0$, pick the smallest 
$i\in\N$ such that $x\in X_i$, so  $x\in X_i\setminus X_{i-1}$ which is a complex manifold of dimension $i$. 
Let $S_i\subset X_i\setminus X_{i-1}$ be the connected component containing $x$. Then the germ of $S_i$ at $x$ is 
contained in every local irreducible component of $X$ at $x$. Since $x$ is a noncritical point of $F|_{S_i}$, it follows from (*) 
that $F$ is strongly noncritical at $x$, thereby proving the claim. By choosing $f_k$ to be strongly noncritical at $p_k\in X_0$ 
we obtain a function $F\in\cO(X)$ that is strongly noncritical on $X$. (To get a strongly noncritical function at a point $p\in X$, 
we can embed $X_p$ as a local complex subvariety of the Zariski tangent space $T_p X\cong\C^N$ and choose a linear function 
on $T_pX$ which is nondegenerate on the tangent space to every local irreducible component of $X$.) 
\qed\smallskip

The proof of Theorem \ref{th:stratified} (cf.\ \S\ref{sec:stratified}) proceeds by induction on the skeleta $X_i$ (\ref{eq:skeleton}). 
The main induction step is furnished by Theorem \ref{th:th1} which provides holomorphic functions on a Stein space 
which have no critical points in the regular locus.
When passing from $X_{i-1}$ to $X_{i}$, we first apply the transversality theorem to show that the 
critical locus of a generic holomorphic extension of a given function on $X_{i-1}$ is discrete and does not accumulate on $X_{i-1}$
(cf.\ Lemma \ref{lem:generic}). We then extend the function to $X_{i}$ without creating any critical points in 
$X_{i}\setminus X_{i-1}$, keeping it fixed to a high order along $X_{i-1}$. To this end we adjust one of the main tools 
from \cite{FF:Acta}, namely {\em the splitting lemma for biholomorphic maps close to the identity} 
on a Cartain pair \cite[Theorem 4.1]{FF:Acta}, to the setting of Stein spaces; see Theorem \ref{th:splitting} 
in \S\ref{sec:gluing} below. 

Besides its original use, this splitting lemma from \cite{FF:Acta} has found a variety of applications. 
In particular, it was used for exposing boundary points of certain classes of pseudoconvex domains, a technique applied in the 
constructions of proper holomorphic embeddings of open Riemann surfaces to $\C^2$ \cite{FW2009,FW2013,Majcen2009}, 
in the construction of complete bounded complex curves in $\C^n$ and minimal surfaces in $\R^3$ \cite{AlF1,AlF3}, 
and in the study of the {\em holomorphic squeezing function} of domains in $\C^n$ \cite{DGZ,DFW,FW2014}. 
We are hoping that Theorem \ref{th:splitting} in this paper will also prove useful for other purposes. 
As explained in Remark \ref{rem:3.2gen}, Theorem \ref{th:splitting} and its proof can be generalized to
the case when the biholomorphic map to be decomposed, and possibly also the underlying Cartan pair, depend on some
additional parameters.

%
%
%  COROLLARIES OF THEOREM ON STRATIFIED NONCRITICAL FUNCTIONS
%
%
We mention a couple of immediate corollaries of Theorem \ref{th:stratified}. 

\begin{corollary}
Let $(X,\Sigma)$ be a stratified Stein space. Given a closed discrete set $P$ in $X$, there exists a 
holomorphic function $F\in\cO(X)$ such that for any stratum $S\in \Sigma$ with $\dim S>0$ we have $\Crit(F|_S)= P\cap S$. 
\end{corollary}

This follows from Theorem \ref{th:stratified} applied to a substratification $\Sigma'$ of $\Sigma$ which contains the given 
discrete set $P$ in the zero dimensional skeleton.

By considering the level sets of a function satisfying Theorem \ref{th:stratified} we obtain the following.

\begin{corollary}
\label{cor:foliation}
Every stratified Stein space $(X,\Sigma)$ admits a holomorphic foliation $\cL=\{L_a\}_{a\in A}$ with closed leaves  such that 
for every stratum $S\in\Sigma$ the restricted foliation $\cF|_S=\{L_a\cap S\}_{a\in A}$ is a nonsingular hypersurface 
foliation on $S$.
\end{corollary}

In the remainder of this introduction we indicate how Theorems \ref{th:main}, \ref{th:mainbis}, and \ref{th:stratified} imply 
results concerning critical loci of holo\-mor\-phic functions on complex manifolds which are obtained by desingularizing Stein spaces. 

The simplest example of this type is obtained by desingularizing a Stein space $Y$ with isolated singular points 
$Y_\sing=\{p_1,p_2,\ldots\}$. Let $\pi:X\to Y$ be a desingularization (cf.\ \cite{AHV,BM,Hironaka}). 
The fiber $E_j= \pi^{-1}({p_j})$ over any singular point of $Y$ is a connected compact complex 
subvariety of $X$ of positive dimension with negative normal bundle in the sense of Grauert  \cite{Grauert:modif}. 
(A local strongly plurisubharmonic function near $p_j\in Y$ pulls back to a function that is strongly plurisubharmonic on a 
deleted neighborhood of $E_j$.) The set $\cE =\pi^{-1}(Y_\sing) =\bigcup_j E_j$ is a complex subvariety of $X$ with compact 
irreducible components of positive dimension, and $\cE$ contains any compact complex subvariety of $X$ without $0$-dimensional
components. Furthermore, we have $\pi_* \cO_X = \cO_Y$ and $\pi\colon X\to Y$ is the {\em Remmert reduction} of $X$ 
\cite{Grauert:modif,Remmert}. If $Y$ has only finitely many singular points then the manifold $X$ is 
{\em $1$-convex} and $\cE$ is the {\em exceptional variety} of $X$ \cite{Grauert:q-convexity}. 
By choosing a noncritical function $g\in \cO(Y)$ furnished by Theorem \ref{th:main}, the function $f=g\circ\pi\in \cO(X)$ 
clearly satisfies $\Crit(f)\subset \cE$. Similary, if  $A$ is a discrete set in $X$ then $\pi(A)$ is discrete in $Y$, and by choosing 
$g\in \cO(Y)$ with $\Crit (g) =\pi(A)$ we get a function $f=g\circ\pi\in \cO(X)$ with $\Crit(f) \setminus \cE= A\setminus \cE$. 
If $A$ intersects every connected component of $\cE$, we have $\Crit(f)=A\cup \cE$. This gives the following corollary.

\begin{corollary}
\label{cor:1convex}
A $1$-convex manifold $X$ with the exceptional variety $\cE$ admits a holomorphic function $f\in \cO(X)$ with 
$\Crit(f)\subset \cE$. Furthermore, given a closed discrete set $A$ in $X$, there exists a function $f\in \cO(X)$ 
with $\Crit(f)=A\cup \cE$.
\end{corollary}

In general we can not find a holomorphic function $f\in \cO(X)$ on a $1$-convex manifold $X$ that is noncritical at 
every point of the exceptional variety $\cE$ of $X$. Indeed, assume that $E$ is a smooth component of $\cE$. 
Since $E$ is compact, the restriction $f|_E$ is constant, so the differential of $f$ vanishes along $E$ in the directions 
tangential to $E$. Hence, if $df_x\ne 0$ for all $x\in E$, the differential defines a nowhere vanishing section of the 
conormal bundle of $E$ in $X$, a nontrivial condition which does not always hold as is seen in the following example.

\begin{example} 
\label{ex1}
Fix an integer $n>1$. Let $X$ be $\C^n$ blown up at the origin, and let $\pi\colon X\to \C^n$ denote the base point projection. The
 exceptional variety is $\cE=\pi^{-1}(0) \cong \C\P^{n-1}$. The conormal bundle of $\cE$ is the line bundle $\cO_{\C\P^{n-1}}(+1)$ 
which does not admit any nonvanishing sections, so $X$ does not admit any noncritical holomorphic functions. On the other hand, the 
function $g(z)=z_1^2+z_2^2+\cdots+z_n^2$ on $\C^n$, with $\Crit(g)=\{0\}$, pulls back to a holomorphic function 
$f=g\circ \pi\in\cO(X)$ with $\Crit(f)=\cE$. Similary, the coordinate function $z_j$ on $\C^n$ pulls back to a holomorphic 
function $z_j\circ \pi=\pi_j$ which is noncritical on $X\setminus \cE\cong \C^n\setminus\{0\}$, and 
\[
	\Crit(\pi_j) = \{[z_1\colon z_2 \colon \cdots \colon z_{n}] \in \cE\colon z_j=0\} \cong \C\P^{n-2}.   
\]
Hence the critical locus may be a proper subvariety of the exceptional variety. 
\qed\end{example}

\begin{problem}
Let $X$ be a $1$-convex manifold. Which closed analytic subsets of its exceptional variety $\cE$ are critical loci of 
holomorphic functions on $X$? 
\end{problem}

Going a step further, recall that a complex space $X$ is said to be {\em holomorphically convex} if for any compact set 
$K\subset X$ its  $\cO(X)$-convex hull 
\[
	\widehat K_{\cO(X)} =\{x\in X\colon |f(x)|\le \sup_K |f|\ \ \forall f\in \cO(X)\}
\]
is also compact. This class contains all $1$-convex spaces, but many more. For example, the total space of any holomorphic 
fiber bundle $X\to Y$ with a compact fiber over a Stein space $Y$ is holomorphically convex.  
By Remmert \cite{Remmert}, every holomorphically convex space $X$ admits a proper holomorphic surjection $\pi\colon X\to Y$ 
onto a Stein space $Y$ such that the (compact) fibers of $\pi$ are connected, $\pi_* \cO_X = \cO_Y$, the map $f\mapsto f\circ \pi$ 
is an isomorphism of $\cO(Y)$ onto $\cO(X)$, and every holomorphic map $X\to S$ to a Stein space $S$ factors through $\pi$. 
If $g\in\cO(Y)$ is a noncritical function on $Y$ furnished by Theorem \ref{th:main}, then the function $f = g\circ \pi\in \cO(X)$ 
is noncritical on the set where $\pi$ is a submersion. What else  could be said? 

Another possible line of investigation is the following. In \cite{FF:Acta} we proved that on any Stein manifold $X$ of dimension 
$n$ there exist $q=\left[ \frac{n+1}{2} \right]$ holomorphic functions $f_1,\ldots,f_q\in \cO(X)$ with pointwise independent 
differentials, i.e., such that $df_1 \wedge df_2\wedge\cdots \wedge df_q$ is a nowhere vanishing holomorphic $(q,0)$-form 
on $X$, and this number $q$ is maximal in general by topological reasons. Furthermore, we have the h-principle for holomorphic 
submersions $X\to\C^q$ any $q<n=\dim X$, saying that every $q$-tuple of pointwise linearly independent continuous $(1,0)$-forms
 can be deformed to a $q$-tuple of linearly independent holomorphic differentials $df_1,\ldots, df_q$. What could be said regarding 
 this problem on Stein spaces? For example:

\begin{problem}
Assume that $X$ is a pure $n$-dimensional Stein space 
and let $q$ be as above. Do there exist functions $f_1,\ldots,f_q\in \cO(X)$ 
such that  $df_1 \wedge df_2\wedge\cdots \wedge df_q$ is nowhere vanishing on 
$X_\reg$? What is the answer if $X$ has only isolated singularities?
\end{problem}

Our methods strongly rely on the fact that the critical locus of a generic holomorphic function on a Stein space is discrete 
(see \S\ref{sec:prel}). If $q>1$ then the set $df_1 \wedge df_2\wedge\cdots \wedge df_q=0$ (if nonempty) is a subvariety of 
complex dimension $\ge q-1>0$, and we do not know how to ensure nonvanishing of this form on a deleted neighborhood of 
a subvariety of $X$ as in the case $q=1$. The problem seems nontrivial even for an isolated singular point of $X$.

%%%%%%%%%%%%%%%%%%%%%%%%%%%%%%%%%%%%%%%%%%%%%%%%%%%%%%%
%%%%%%%%%%%%%%%%%%%%
%																				%	
%																				%
%  Preliminaries										           					       		 %
%																				%	
%																				%
%%%%%%%%%%%%%%%%%%%%%%%%%%%%%%%%%%%%%%%%%%%%%%%%%%%%%%%
%%%%%%%%%%%%%%%%%%%%

\section{Critical points of a holomorphic function on a complex space}   \label{sec:prel}

We begin by recalling certain basic facts of complex analytic geometry. 

Let $(X,\cO_X)$ be a reduced complex space. Following standard practice we shall simply write $X$ in the sequel. We denote by 
$\cO(X)\cong \Gamma(X,\cO_X)$ the algebra of all holomorphic functions on $X$. Given a holomorphic function $f$ on an open 
set $U\subset X$, we denote by $f_p\in \cO_{X,p}$ the germ of $f$ at a point $p\in U$. Similarly, $X_p$ stands for the germ of 
$X$ at a point $p\in X$. 

By $\mgot_p=\mgot_{X,p}$ we denote the maximal ideal of the local ring $\cO_{X,p}$, so $\cO_{X,p}/\mgot_p\cong\C$. 
We say that $f\in\cO_{X,p}$ {\em vanishes to order $k \in \N$ at the point $p$} if $f\in\mgot_p^{k}$ (the $k$-th power of the
 maximal ideal). The quotient ring $\cO_{X,p}/\mgot_p^{k}\cong \C \oplus \mgot_p/\mgot_p^{k}$ is a finite dimensional complex
vector space, called the {\em space of $(k-1)$-jets of holomorphic functions on $X$ at $p$}. Recall that  
$\mgot_p/\mgot_p^2\cong T_p^*X$ is the Zariski cotangent space and its dual $(\mgot_p/\mgot_p^2)^* \cong T_p X$ is the 
Zariski tangent space of $X$ at $p$.

If $X'$ is a complex subvariety of $X$ and $p\in X'$, then the maximal ideal $\mgot_{X',p}$ of the ring $\cO_{X',p}$ 
consists of all germs at $p$ of restrictions $f|_{X'}$ with $f\in \mgot_{X,p}$.

\begin{lemma} \label{lem:germs}
Let $X'$ be a closed complex subvariety of a complex space $X$ and $p\in X'$. If $f\in \cO_{X',p}$ and $h\in \cO_{X,p}$ 
are such that $f-(h|_{X'})_p \in \mgot_{X',p}^{k}$ for some $k\in\N$, then there exists $\tilde h \in \cO_{X,p}$ 
such that $\tilde h -h\in \mgot_{X,p}^{k}$ and $(\tilde h|_{X'})_p = f \in \cO_{X',p}$.  
\end{lemma}

\begin{proof}
The conditions imply that $f = (h|_{X'})_p + \sum_j \xi_{j,1}\xi_{j,2}\cdots \xi_{j,k}$ where $\xi_{j,i}\in \mgot_{X',p}$ 
for all $i$ and $j$. Then $\xi_{j,i}=\tilde \xi_{j,i}|_{X'}$ for some $\tilde \xi_j\in\mgot_{X,p}$, and the germ 
$\tilde h = h +\sum_j \tilde \xi_{j,1}\tilde \xi_{j,2}\cdots \tilde \xi_{j,k} \in\cO_{X,p}$ satisfies the stated properties.
\end{proof}

Given a function $f\in\cO(X)$, the collection of its differentials $df_x\colon T_x X\to\C$ over all points $x\in X$ defines the 
{\em tangent map} $Tf\colon TX\to X\times \C$ on the {\em tangent space}  $TX=\bigcup_{x\in X} T_x X$. Recall that 
$TX$ carries the structure of a not necessarily reduced {\em linear space} over $X$ such that the tangent map $Tf$ is holomorphic. 
Here is a local decription of $TX$ (see e.g.\ \cite[Chapter 2]{Fischer}). Assume that $X$ is a closed complex subvariety of an open 
set $U\subset \C^N$, defined by holomorphic functions $h_1,\ldots,h_m \in \cO(U)$ which generate the sheaf of ideals $\cJ_X$ 
of $X$ (hence $\cO_{X}\cong (\cO_U/\cJ_X)|_X$). Let $(z_1,\ldots,z_N,\xi_1,\ldots,\xi_N)$ be complex coordinates on 
$U\times \C^N$. Then $TX$ is the closed complex subspace of $U\times \C^N$ generated by the functions
\begin{equation}
\label{eq:TX}
	h_1,\ldots,h_m\ \ {\rm and}\ \
   \frac{\partial h_i}{\partial z_1}\,\xi_1 + \cdots + \frac{\partial h_i}{\partial z_N}\,\xi_N
   \ \ {\rm for}\ \ i=1,\ldots,m.                      
\end{equation}
This means that $TX$ is the common zero set of the above functions and it structure sheaf $\cO_{TX}$ is the quotient of 
$\cO_{U\times \C^N}$ by the ideal generated by them. 
The projection $TX\to X$ is the restriction of the projection $U\times \C^N\to U$, $(z,\xi)\mapsto z$. Different local representations 
of $X$ give isomorphic representations of $TX$. If $X$ is a complex manifold then $TX$ is the usual tangent bundle of $X$; 
this holds in particular over the regular locus $X_{\reg}$ of any complex space. 

Since the critical locus $\Crit(f)$ of a holomorphic function $f\in\cO(X)$ is the set of points $x\in X$ at which the differential 
$df_x\colon T_x X\to\C$ vanishes, one might expect that $\Crit(f)$ is a closed complex subvariety of $X$. This is clearly true if
 $X$ is a complex manifold (in particular, it holds on the regular locus $X_\reg$ of any complex space), but it fails in general 
 near singularities. Furthermore, unlike in the smooth case, the set of (strongly) noncritical holomorphic functions is not stable 
 under small deformations. The following examples illustrates these phenomena in a simple setting of an irreducible quadratic 
 surface singularity in $\C^3$.

\begin{example} \label{ex:null}
Let $A$ be the subvariety of $\C^3$ given by 
\begin{equation}
\label{eq:null}
	A=\{(z_1,z_2,z_3) \in \C^3 \colon h(z)=z_1^2+z_2^2+z_3^2=0\}.
\end{equation}
(In the theory of minimal surfaces this is called the {\em null quadric}, and a complex curve in $\C^3$ whose  derivative belongs 
to $A^*=A\setminus \{(0,0,0)\}$ is said to be a (immersed) {\em null holomorphic curve}. Such curves are related to conformally 
immersed minimal surfaces in $\R^3$. See e.g.\ \cite{Osserman} for a classical survey of this subject and \cite{AlF3} for some 
recent results.) Clearly $A_\sing =\{(0,0,0)\}$, $A$ is locally and globally irreducible, and $T_{(0,0,0)}A=\C^3$. For any 
$\lambda=(\lambda_1,\lambda_2,\lambda_3)\in\C^3\setminus\{(0,0,0)\}$ the linear function 
\[
   f_\lambda (z_1,z_2,z_3)= \lambda_1 z_1 + \lambda_2 z_2 + \lambda_3 z_3
\]
restricted to $A$ is strongly noncritical at $(0,0,0)$. Clearly $df_\lambda$ is colinear with $dh=2(z_1dz_1+z_2dz_2+z_3dz_3)$ 
precisely along the complex line $\Lambda=\C\lambda= \{t\lambda\colon t\in \C\}$. If $\lambda\in A^*$, it follows that 
$\Crit({f_\lambda}|_A)=\Lambda\setminus \{0\}$ which is not closed. An explicit example is obtained by taking 
\[
	\lambda=(1,\imath,0) \in A^*,\qquad f(z)=z_1+\imath z_2. 
\]
(Here $\imath=\sqrt{-1}$.) 
\qed\end{example}

Let us now show on the same example that the set of (strongly) noncritical functions fails to be stable under small deformations.

\begin{example}
\label{ex:null2}
Let $A$ be the quadric (\ref{eq:null}). Consider the family of functions
\[
	f_\epsilon(z_1,z_2,z_3)= z_1+ z_1(z_1-2\epsilon) + \imath z_2 ,\qquad \epsilon\in\C.
\]
Since $(df_\epsilon)_0=(1-2\epsilon)dz_1+\imath dz_2$, $f_\epsilon|_A$ is (strongly) noncritical at the origin for any $\epsilon\in\C$. 
A calculation shows that for $\epsilon\ne 1/2$ the differentials $df_\epsilon$ and $dh$ (considered on the tangent bundle $T\C^3$)
are colinear precisely at points of the complex curve 
\[
	C_\epsilon =\{(z_1,z_2,0)\in\C^3: z_2=\imath z_1/(2z_1-2\epsilon+1)\}. 
\]
This curve intersects the quadric $A$ at the following four points:
\[
		A\cap C_\epsilon =\{(0,0,0), (\epsilon,\imath\epsilon,0), (\epsilon-1,-\imath(\epsilon-1),0)\}.
\]
Hence the second and the third of these points are the critical points of $f_\epsilon|_A$ when $ \epsilon\notin\{0,1\}$. 
For $\epsilon$ close to $0$ the point $(\epsilon,\imath\epsilon,0)$ lies close to the origin, while the third point is close to 
$(-1,\imath,0)$.  Hence the function $f_0|_A$ is noncritical on the intersection of $A$ with the ball of radius $1/2$ 
around the origin in $\C^3$, but $f_\epsilon|_A$ for small $\epsilon\ne 0$ is close to $f_0$ and has a critical point 
$(\epsilon,\imath\epsilon,0)\in A$ near the origin.
\qed\end{example}

Although we have seen in Example \ref{ex:null} that $\Crit(f)$ need not be a closed complex subvariety near singular 
points of a complex space, we still have the following result.

\begin{lemma}
\label{lem:crit}
Let $f$ be a holomorphic function on a complex space $X$. If $X'\subset X$ is a closed complex subvariety of $X$ 
containing the singular locus $X_\sing$ of $X$, then the set
\[
			C_{X'}(f) := \{x\in X_\reg\colon df_x=0\} \cup {X'}
\]
is a closed complex subvariety of $X$.
\end{lemma}

\begin{proof}
By the desingularization theorem \cite{AHV,BM,Hironaka} there are a complex manifold $M$ and a proper holomorphic surjection 
$\pi\colon M\to X$ such that $\pi\colon M\setminus \pi^{-1}(X_{\rm sing})\to X\setminus X_{\rm sing}$ is a biholomorphism and 
$\pi^{-1}(X_{\rm sing})$ is a compact complex hypersurface in $M$. Given a function $f\in\cO(X)$, consider the function 
$F=f\circ\pi\in\cO(M)$ and the subvariety $M'=\pi^{-1}(X')$ of $M$. Since $M$ is a complex manifold, the critical locus 
$\Crit(F)\subset M$ is a closed complex subvariety of $M$, and hence so is the set $C_{M'}(F)=\Crit(F)\cup M'$. 
As $\pi$ is proper, $\pi(C_{M'}(F))$ is a closed complex subvariety of $X$ according to Remmert \cite{Remmert2}. 
Since $\pi$ is biholomorphic over $X_\reg$, we have that $\pi(C_{M'}(F))=C_{X'}(f)$ which proves the result. 
\end{proof}

In spite of the lack of stability of noncritical functions, illustrated by Example \ref{ex:null2}, we shall obtain a certain stability 
result (cf.\ Lemma \ref{lem:stability} below) which will be used in the construction of stratified noncritical holomorphic functions 
on Stein spaces.

Given a compact set $K$ in a complex space $X$, we denote by $\cO(K)$ the space of all functions $f$ that are holomorphic on 
an open neighborhood $U_f \subset X$ of $K$ (depending on the function), identifying two function that agree on some neighborhood 
of $K$. By $\mathring K$ we denote the topological interior of a set $K$.

For any coherent analytic sheaf $\cF$ on a complex space $X$ the $\cO(X)$-module $\cF(X)=\Gamma(X,\cF)$ of all global 
sections of $\cF$ over $X$ can be endowed with a Fr\'echet space topology (the topology of uniform convergence on compacts 
in $X$) such that for every point $x\in X$ the natural restriction map $\cF(X)\mapsto \cF_x$ is continuous 
(see Theorem 5 in \cite[p.\ 167]{Grauert-Remmert1979}). The topology on the stalks $\cF_x$ is the 
{\em sequence topology} (cf.\ \cite[p.\ 86ff]{GR-Stellenalgebren}). Thus every set of the second category in $\cF(X)$ 
(an intersection of at most countably many open dense sets) is dense in $\cF(X)$. The expression {\em generic holomorphic function}
 on $X$ will always mean a function in a certain set of the second category in $\cO(X)$, and likewise for $\cF(X)$. 

If $\cS$ is a coherent subsheaf of a coherent sheaf $\cF$ over $X$ then $\cS(X)$ is a closed submodule of $\cF(X)$ 
(the Closedness Theorem, cf.\ \cite[p.\ 169]{Grauert-Remmert1979}). Since every $\cO_{X,x}$-submodule $M$ of the module 
$\cF_x$ is closed in the sequence topology, it follows that $\{f \in \cF(X)\colon f_x\in M\}$ is a closed subspace of $\cF(X)$, 
hence a Fr\'echet space. 

In particular, if $X'$ is a closed complex subvariety of a complex space $X$ and $\cJ_{X'}$ is the sheaf of ideals of $X'$ 
(a coherent subsheaf of $\cO_X$), then 
\[
	\cJ(X'):= \Gamma(X,\cJ_{X'}) = \{f\in \cO(X)\colon f|_{X'}=0\}
\]
is a closed (hence Fr\'echet) ideal in $\cO(X)$. Given a function $g\in \cO(X')$ on a closed complex subvariety $X'\subset X$, the set
\begin{equation} \label{eq:Xprimeg}
	\cO_{X',g}(X)=\{f\in \cO(X)\colon f|_{X'}=g\}
\end{equation}
is a closed affine subspace of $\cO(X)$ and hence a Baire space. 

The Closedness Theorem \cite[p.\ 169]{Grauert-Remmert1979} shows that for any point $x\in X$ and $k\in\N$ the set 
\[
	\{f\in\cO(X)\colon f_x - f(x) \in \mgot_x^k\}
\]
is closed in $\cO(X)$. For $k=2$ this is the set of functions with a critical point at $x$. 

Let $X_x=\bigcup_{j=1}^m V_j$ be a decomposition into local irreducible components at a point $x\in X$. 
According to Definition \ref{def:critical}, a function $f\in\cO_{X,x}$ fails to be strongly noncritical at $x$ if there is a 
$j\in\{1,\ldots,m\}$ such that $(f|_{V_j})_x-f(x)\in \mgot_{V_j,x}^2$. This defines a closed subset of $\cO_{X,x}$, 
so the set of all strongly noncritical germs is open in $\cO_{X,x}$. Since the restrictions maps in the space of sections 
of a coherent sheaf are continuous, we get the following conclusion.

\begin{lemma}\label{lem:stability0}
The set of all functions $f\in\cO(X)$ which are noncritical (or strongly noncritical) at a certain point $x\in X$ is open in $\cO(X)$.    
\end{lemma}

However, Example \ref{ex:null2} above shows that the set of functions $f\in\cO(X)$ that are noncritical (or strongly noncritical) 
on a certain compact set $K\subset X$ may fail to be open in $\cO(X)$, unless $K$ is contained in the regular locus $X_\reg$.

The following result is \cite[Lemma 3.1, p.\ 52]{FP3} in the case that $X$ is a Stein manifold; 
we shall need it also when $X$ is a Stein space. (We correct a misprint in the original source.)
%
%  Bounded extension operator.
%
\begin{lemma}[{\bf Bounded extension operator}] 
\label{lem:FP3}
Let $X$ be a Stein space, $X'$ be a closed complex subvariety of $X$, and $\Omega\Subset X$ be a Stein domain
in $X$. For any relatively compact subdomain $D\Subset \Omega$ there exists a bounded linear extension
operator $T\colon \cH^\infty(\Omega\cap X') \to \cH^\infty(D)$ such that 
\[
	(Tf)(x)=f(x)\quad \forall f\in  \cH^\infty(\Omega\cap X'),\ \forall x\in D\cap X'.
\]
\end{lemma}

\begin{proof}
Choose a Stein neighborhood $W\Subset X$ of the compact set $\overline\Omega$ and embed it as a closed 
complex subvariety  (still denoted $W$) of some Euclidean space $\C^N$
(see \cite[Theorem 2.2.8]{FF:book} and the references therein). 
By Siu's theorem \cite{Siu1976} there is a Stein domain $\Omega'\Subset \C^N$ such that $\Omega= \Omega'\cap W$. 
Also choose a domain $D'$ in $\C^N$ such that $D\subset D'$ and $\overline D'\subset \Omega'$. 
By \cite[Lemma 3.1]{FP3}, applied with the subvariety $X'\cap W$ 
of the Stein manifold $\C^N$ and domains $D'\Subset \Omega' \Subset \C^N$, there exists a bounded linear extension operator
$T'\colon \cH^\infty(\Omega'\cap X') \to \cH^\infty(D')$. Since $\Omega'\cap W=\Omega$, we obtain by restricting the 
resulting funtion $T'f$ to $D\subset W\cap D'$ a bounded extension operator $T$ as in the lemma.
\end{proof}

\begin{lemma}[{\bf The Stability Lemma}]
\label{lem:stability}
Assume that $X$ is a complex space, $X'\subset X$ is a closed complex subvariety containing $X_\sing$, 
and $K\subset L$ are compact subsets  of $X$ with $K\subset \mathring L$. 
Assume that $f\in\cO(X)$ is noncritical on $L\setminus X'$.  Then there exist an integer 
$r\in \N$ and a number $\epsilon>0$ such that the following holds. If a function $g\in \cO(L)$ satisfies
the conditions
\begin{itemize}
\item[\rm (i)] $f-g \in \Gamma(L,\cJ^r_{X'})$, where $\cJ_{X'}^r$ is the $r$-th power of the ideal sheaf $\cJ_{X'}$, and
\item[\rm (ii)] $||f-g||_L := \sup_{x\in L}|f(x)-g(x)| <\epsilon$,
\end{itemize}   
then $g$ has no critical points on $K\setminus X'$.
\end{lemma}

\begin{proof}
The result holds on compact subsets of $X\setminus X' \subset X_\reg$ in view of Lemma \ref{lem:stability0}, 
so it suffices to consider the behavior of $g$ near $K\cap X'$. 

Fix a point $p\in K \cap X'$ and embed an open neighborhood $U\subset X$ of $p$ as a closed complex subvariety 
(still denoted $U$) of an open ball $B \subset\C^N$. We choose $U$ small enough such that $U\subset \mathring L$. 
Pick a slightly smaller ball $B'\Subset B$ and set $U':=B'\cap U$. Lemma \ref{lem:FP3} (applied with the domain $\Omega=B$ 
in $X=\C^N$, the subvariety $X'=U$, and the subdomain $D=B'\Subset B$) furnishes a bounded linear extension operator 
$T$ mapping bounded holomorphic functions on $U$ to bounded holomorphic functions on $B'$. In the embedded picture, 
a point $x\in U' \setminus X' \subset B'$ is a critical point of $f$ if and only if the differential $d\tilde f_x \colon T_x\C^N\to\C$ 
of the extended function $\tilde f = Tf \in\cO(B')$ annihilates the Zariski tangent space $T_x U$. The latter condition is expressed 
by a finite number of holomorphic equations $F_j(f)=0$ on $B'$ $(j=1,\ldots,k)$ involving the values and first order partial
derivatives of $\tilde f$ and of some fixed  holomorphic defining functions $h_1,\ldots, h_m$ for the subvariety $U$ in $B$. 
(These equations express the fact that the differential of $\tilde f$ is contained in the linear span of the differentials of 
the functions $h_1,\ldots, h_m$; compare with the local description (\ref{eq:TX}) of $TX$.) By the assumption this system 
of equations has no solutions on $U\setminus X'$. If a bounded function $g\in \cO(U)$ agrees with $f$ to order $r$ along 
the subvariety $U \cap X'$ then, setting $\tilde g=Tg\in  \cO(B')$, the corresponding functions $F_j(g)|_{U'}$ agree with 
the functions $F_j(f)|_{U'}$ to order $r-1$ along $U'\cap X'$. By choosing the integer $r\in\N$ sufficiently big and the number 
$\epsilon$ bounded from above by some fixed number $\epsilon_0>0$, we can ensure that for any $g$ satisfying conditions 
(i) and (ii) the following system of holomorphic equations on $B' \subset \C^N$,
\begin{equation}
\label{eq:zeroset}
		h_1=0,\ldots, h_m=0,\quad F_1(g)=0,\ldots, F_k(g)=0, 
\end{equation}
has no solutions in $W\setminus X'$, where $W\subset U'$ is a neighborhood of $p$ whose size depends on $r$ and $\epsilon_0$.
 (This essentially follows from the \L ojasiewicz inequality, see e.g.\ \cite{JKS}. The details of this argument can also found in 
 the proof of \cite[Theorem 1.3]{FF:CI}; see in particular pp.\ 507--509.
In fact, looking at the common zero set of the system (\ref{eq:zeroset}) as the inverse image of the origin $0\in \C^{m+k}$ 
by the holomorphic map $B'\to \C^{m+k}$ whose components are the functions in (\ref{eq:zeroset}), the local aspect of the 
cited result from \cite{FF:CI} applies verbatim.) Since finitely many open sets $U'$ of this kind cover $K\cap X'$, we see 
that the system (\ref{eq:zeroset}) has no solutions on a deleted neighborhood of $K\cap X'$ in $K$. By choosing $\epsilon>0$ 
small enough we can also ensure in view of Lemma \ref{lem:stability0} that there are no solutions on the rest of $K$.
\end{proof}

Lemma \ref{lem:stability} fails in general without the interpolation condition as shown by Example \ref{ex:null2}. 
Here is an even simpler example on the cusp curve, showing that being {\em stratified noncritical} 
(see Definition \ref{def:stratified-noncritical}) is not a stable property even on complex curves if we allow 
critical points in the zero dimensional skeleton.

\begin{example}
\label{ex:nonstable}
The cusp curve $X=\{(z,w)\in \C^2\colon z^2=w^3\}$ has a singularity at the origin $(0,0)\in \C^2$ and is smooth elsewhere. 
It is desingularized by the map $\pi\colon \C\to X$, $\pi(t)=(t^3,t^2)$. The function $f(z,w)=zw$ on $X$ pulls back to the 
function $h(t)=f(\pi(t))=t^5$ with the only critical point at $t=0$, so $f|_X$ is stratified noncritical with respect to 
$\{(0,0)\} \subset X$. The perturbation of $h$ given by
\[
 	h_\epsilon(t)=t^3(t-\epsilon)^2= t^5 - 2\epsilon t^4 + \epsilon^2 t^3 = zw-2\epsilon w^2 + \epsilon^2 z
\] 
induces a holomorphic function $f_\epsilon \colon X\to \C$ with a critical point at $(\epsilon^3,\epsilon^2)$, 
so $f_\epsilon$ is not stratified noncritical on $\{(0,0)\} \subset X$ if $\epsilon\ne 0$.   	
\qed\end{example}

\begin{lemma}[{\bf The Genericity Lemma}]
\label{lem:generic}
Let $X$ be a Stein space. 
\begin{itemize}
\item[\rm (i)] For a generic $f\in \cO(X)$ the set $A(f):=\Crit(f|_{X_\reg})$ is discrete in $X$. 
\item[\rm (ii)] If $X'\subset X$ is a closed complex subvariety containing $X_\sing$ and $g\in \cO(X')$, then for a generic 
$f\in \cO_{X',g}(X)$ 	%(\ref{eq:Xprimeg})
 the set $\Crit(f|_{X\setminus X'})$ is discrete in $X$. In particular, 
a generic holomorphic extension of $g$ is noncritical on a deleted neighborhood of $X'$ in $X$. 
\item[\rm (iii)] If $g$ is a holomorphic function on an open neighborhood of $X'$ in $X$ and $r\in\N$, 
then the conclusion of part (ii) holds for a generic extension $f\in\cO(X)$ of $g|_{X'}$ which agrees with $g$ 
to order $r$ along $X'$.
\end{itemize}
\end{lemma}

\begin{proof}
We begin by proving part (i). A point $x\in X_\reg$ is a critical point of a holomorphic function $f\in\cO(X)$ if and only 
if the partial derivatives $\partial f/\partial z_j$ in any system of local holomorphic coordinates $z=(z_1,\ldots,z_n)$ on an 
open neighborhood $U\subset X_\reg$ of $x$ vanish at the point $z(x)$. (Here $n=\dim_x X$.) This gives $n$ independent
holomorphic equations on the 1-jet extension $j^1 f$ of $f$, so the jet transversality theorem for holomorphic maps $X\to\C$ 
(cf.\ \cite{Forster1970} or \cite[\S 7.8]{FF:book}) implies that every point $x\in A(f)$ is an isolated point of $A(f)$ for a generic 
$f\in \cO(X)$. (The argument goes as follows: write $X_\reg=\bigcup_{j=1}^\infty U_j$ where $U_j\subset X_\reg$ is a 
compact connected coordinate neighborhood for every $j$. The set of all functions $f\in \cO(X)$ whose 1-jet extension 
$U_j\ni x\mapsto j^1_x f\in \C^{n_j}$ (with $n_j=\dim U_j$) is transverse to $0\in\C^{n_j}$ on the compact set $U_j$
is open and dense in $\cO(X)$. Taking the countable intersection of these sets over all $j$ gives the statement.) 

For any  function $f$ as above
the set $A(f)$ is discrete in $X_\reg$, and we claim that $A(f)$ is then also discrete in $X$. If not, there is a point 
$x_0\in X_\sing$ and a sequence $x_j\in A(f)$ with $\lim_{j\to\infty} x_j=x_0$. By Lemma \ref{lem:crit} the set 
$C(f)=A(f)\cup X_\sing$ is a closed complex subvariety of $X$. Pick a compact neighborhood $K\subset X$ of $x_0$. 
Each point $x_j$ from the above sequence which belongs to $K$ is an isolated point of $C(f)$, hence an irreducible 
component of $C(f)$. Thus the compact subset $K\cap C(f)$ of the complex space $C(f)$ contains infinitely many 
irreducible components of $C(f)$, a contradiction. This proves part (i).

Part (ii) follows similarly by applying the jet transversality theorem in the Baire  space 
$\cO_{X',g}(X)=\{f\in \cO(X)\colon f|_{X'}=g\}$. 

Finally, let $g$ be as in (iii). Consider the short exact sequence of coherent analytic sheaves 
$0\to \cJ_{X'}^{r}\to \cO_X\to \cO_X/\cJ_{X'}^r\to 0$. The sheaf $\cO_X/\cJ_{X'}^r$ is supported on $X'$, and hence 
$g$ determines a section of it. Since $H^1(X;\cJ_{X'}^{r})=0$ by Cartan's Theorem B, the same section is determined by 
a function $G\in\cO(X)$. This means that $G-g$ vanishes to order $r$ along $X'$. To conclude the proof, 
it suffices to apply the transversality theorem in the Baire space $G+\cJ_{X'}^{r}(X) \subset \cO(X)$; 
the details are similar as in part (i). 
\end{proof}

\begin{proposition}
\label{prop:generic2}
If $(X,\Sigma)$ is a a stratified Stein space, then the set $\bigcup_{S\in\Sigma} \Crit(f|_S)$ is discrete in $X$ for a 
generic $f\in \cO(X)$. 
\end{proposition}

\begin{proof}
Let $\Sigma=\{S_j\}_j$ where $S_j$ are (smooth) strata. Each stratum $S_j$ of positive dimension $n_j>0$ is a union 
$S_j=\bigcup_k U_{j,k}$ of countably many compact coordinate sets $U_{j,k}$. The same argument as in the proof of 
Lemma \ref{lem:generic} shows that set $\cU_{j,k}\subset \cO(X)$, consisting of all $f\in \cO(X)$ such that the 1-jet 
extension map $U_{j,k} \ni x\mapsto j^1_x f \in \C^{n_j}$ is transverse to $0\in\C^{n_j}$ on $U_{j,k}$, is open and dense in 
$\cO(X)$. Every $f\in \bigcap_{j,k} \cU_{j,k}$ satisfies the conclusion of the proposition. 
\end{proof}

Since every complex space admits a stratification, Proposition \ref{prop:generic2} implies

\begin{corollary}
\label{cor:generic}
A generic holomorphic function on a Stein space has discrete critical locus.
\end{corollary}

We also have the following result in which $X$ is not necessarily Stein.

\begin{corollary}
\label{cor:deleted}
Let $X$ be a complex space and $X'\subset X$ a closed Stein  subvariety containing $X_\sing$. 
Given a function $g\in \cO(X')$, there are an open neighborhood $U\subset X$ of $X'$ and a 
function $f\in\cO(U)$ such that $f|_{X'}=g$ and $f$ has no critical points in $U\setminus X'$. 

In particular, an isolated sigular point $p$ of a complex space $X$ admits a holomorphic function on a neighborhood 
$U$ of $p$ which is noncritical on $U\setminus\{p\}$.
\end{corollary}

\begin{proof}
According to Siu \cite{Siu1976} (see also \cite[\S 3.1]{FF:book} and the additional references therein) a Stein subvariety $X'$ 
in any complex space $X$ admits an open Stein neighborhood $\Omega \subset X$ containing $X'$ as a closed complex 
subvariety. The conclusion then follows from Lemma \ref{lem:generic} applied to the Stein space $\Omega$. 
\end{proof}

In the proof of Theorem \ref{th:splitting} we shall also need the following result.
This  is well known when $X$ is a complex manifold (i.e., without singularities), and we shall 
reduce the proof to this particular case. 

\begin{lemma}
\label{lem:closetoId}
Let $X$ be a reduced complex space and $U\Subset U'$ open relatively compact sets in $X$. Fix a distance function $\dist$ on 
$X$ inducing the standard topology. There is a constant $\epsilon>0$ such that for any holomorphic map $f:U' \to X$ satisfying 
$\sup_{x\in U'}\dist(x,f(x))<\epsilon$ the restriction $f|_U\colon U\to f(U) \subset X$ is biholomorphic onto its image. 
\end{lemma}

\begin{proof}
We first prove the lemma in the case when $U'$ is Stein and its (compact) closure $\overline U'$ admits a Stein neighborhood 
$W$ in $X$.  Assuming as we may that $W$ is relatively compact,  it embeds as a closed 
complex subvariety of a Euclidean space $\C^N$ (cf.\ \cite[Theorem 2.2.8]{FF:book}). Since  $U'$ is 
Stein, Siu's theorem \cite{Siu1976} provides a bounded Stein domain $D'\Subset \C^N$ such that $D'\cap W=U'$.
Choose a pair of domains $D_0\Subset D$ in $\C^N$ such that $\overline U \subset D_0 \cap W$ and $\bar D\subset D'$. 
Let $T$ be a bounded linear extension operator furnished by Lemma \ref{lem:FP3}, mapping bounded holomorphic functions 
on $U'$ to bounded holomorphic functions on $D$ and satisfying 
\[
	Tg|_{D\cap U'}= g|_{D\cap U'} \quad \text{and} \quad ||Tg||_D \le C||g||_{U'}
\]
for some constant $C>0$ independent of $g$. 

Consider a holomorphic map $f\colon U'\to X$ close to the identity. We may assume that $f(U')\subset W \subset \C^N$. 
Write $f(x)=x+g(x)$ for $x\in U'$, where $g:U'\to\C^N$ is close to zero. Applying the operator $T$ to each component of $g$ 
we get a holomorphic map $F=\Id+Tg \colon D\to \C^N$ which is close to the identity in the sup norm on $D$. Hence $F$ is
biholomorphic on the smaller domain $D_0$ provided that $f$ is close enough to the identity on $U'$. Since $U\subset D_0$ 
and $F|_{U}=\Id_U + g|_U=f|_U$, we infer that $f\colon U\to f(U) \subset W$ is biholomorphic as well. Furthermore, the inverse 
map $F^{-1}\colon F(D_0)\to D_0$, restricted to $F(D_0)\cap W$, has range in $W$ as is easily seen by considering the situation 
on $W_\reg$ and applying the identity principle. This completes the proof in the special case.

The general case follows as in the standard manifold situation. By compactness of $\overline U$ 
we can choose finitely many triples of open sets $V_j\Subset U_j\Subset U'_j$ in $X$ $(j=1,\ldots,m)$ such that 
\begin{itemize}
\item[\rm (i)]   $\overline U\subset \bigcup_{j=1}^m V_j$ and $\bigcup_{j=1}^m U'_j \subset U'$, and
\item[\rm (ii)]  $U'_j$ is Stein and $\overline U'_j$ has a Stein neighborhood in $X$ for every $j=1,\ldots,m$.
\end{itemize}
Pick a number $\epsilon_0>0$ such that $\dist(V_j,X\setminus U_j) > 2\epsilon_0$ for every $j=1,\ldots,m$. By the special 
case proved above, applied to the pair $U_j\Subset U'_j$, we can find a number $\epsilon\in (0,\epsilon_0)$ such that 
$f|_{U_j}\colon U_j\to f(U_j)$ is biholomorphic for every $j$ provided that $\dist(x,f(x))<\epsilon$ for all $x\in U'$. 
Since $U\subset \bigcup_{j=1}^m U_j$, it follows that $f|_U:U\to f(U)$ is biholomorphic as long as it is injective. 
Suppose that $f(x)=f(y)$ for a pair of point $x\ne y$ in $U$. Since the sets $V_j$ cover $U$, we have $x\in V_j$ for 
some $j$. As $f$ is injective on $U_j$, it follows that $y\in U\setminus U_j$ and hence $\dist(x,y) >2\epsilon_0$. 
The triangle inequality and the choice of $\epsilon$ then gives 
\[
	  \dist(f(x),f(y)) \ge \dist(x,y) - \dist(x,f(x)) - \dist(y,f(y))>2\epsilon_0-2\epsilon>0,
\]
a contradiction to $f(x)=f(y)$. Thus $f$ is injective on $U$. 
\end{proof}

%%%%%%%%%%%%%%%%%%%%%%%%%%%%%%%%%%%%%%%%%%%%%%%%%%%%%%%
%%%%%%%%%%%%%%%%%%%%%%%%%%%%%%%%%%%%%%%%%%%%%%%%%%%%%%%
%																				%	
%																				%
%  A splitting lemma for biholomorphic maps on complex spaces           					        			%
%																				%	
%																				%
%%%%%%%%%%%%%%%%%%%%%%%%%%%%%%%%%%%%%%%%%%%%%%%%%%%%%%%
%%%%%%%%%%%%%%%%%%%%%%%%%%%%%%%%%%%%%%%%%%%%%%%%%%%%%%%

\section{A splitting lemma for biholomorphic maps on complex spaces} \label{sec:gluing}

In this section we prove a splitting lemma for biholomorphic maps close to the identity on  Cartan pairs in complex spaces; 
see Theorem \ref{th:splitting} below. This result is the key to the proof of our main theorems; 
it will be used for gluing pairs of holomorphic functions with control of their critical loci. 
The nonsingular case is given by \cite[Theorem 4.1]{FF:Acta}.

Recall that a compact set $K$ in a complex space $X$ is said to be a {\em Stein compact} if $K$ admits a basis of open 
Stein neighborhoods in $X$. We recall the following notion.

%
%
%  Definition of a Cartan pair
%
%
\begin{definition} \label{def:CP}
{\rm \cite[p.\ 209]{FF:book}}
{\rm (I)} A pair $(A,B)$ of compact subsets in a complex space $X$ is a {\em Cartan pair} if it satisfies the following conditions:  
\begin{itemize}
\item[\rm(i)]   $A$, $B$, $D=A\cup B$ and $C=A\cap B$ are Stein compacts, and
\item[\rm(ii)]  $A,B$ are {\em separated} in the sense that
$\overline{A\setminus B}\cap \overline{B\setminus A} =\emptyset$.
\end{itemize}
\noindent {\rm (II)} A pair $(A,B)$ of open sets in a complex manifold $X$ is a {\em strongly pseudoconvex Cartan pair of 
class $\cC^\ell$} $(\ell\ge 2)$ if $(\bar A,\bar B)$ is a Cartan pair in the sense of (I) and the sets $A$, $B$, $D=A\cup B$ 
and $C=A\cap B$ are Stein domains with strongly pseudoconvex boundaries of class $\cC^\ell$.
\end{definition}

We shall use the following properties of Cartan pairs:
\begin{itemize}
\item[\rm (a)] Let $(A,B)$ be a Cartain pair in $X$. If $X$ is a complex subspace of another complex space 
$\wt X$, then $(A,B)$ is also a Cartan pair in $\wt X$ (cf.\  \cite[Lemma 5.7.2, p.\ 210]{FF:book}).
\item[\rm (b)] Every Cartan pair $(A,B)$ in a complex manifold $X$ can be approximated from outside by smooth 
strongly pseudoconvex Cartain pairs (cf.\ \cite[Proposition 5.7.3, p.\ 210]{FF:book}). 
\item[\rm (c)] One can solve any Cousin-I problem with sup-norm bounds on a strongly pseudoconvex Cartan pair 
(cf.\ \cite[Lemma 5.8.2, p.\ 212]{FF:book}).
\end{itemize}

We denote by $\dist$ a distance function on $X$ which induces its standard complex space topology. 
(The precise choice will not be important.) Given a compact set $K\subset X$ and continuous maps $f,g\colon K\to X$, 
we shall write 
\[
	\dist_K(f,g)=\sup_{x\in K} \dist(f(x),g(x)).
\]
By $\Id$ we denote the identity map; its domain will always be clear from the context.

%%%%%%%%%%%%%%%%%%%%%%%%%%%%%%%%%%%%%%%%%%%%%%%%%%%%%
%
%
%    Splitting lemma for biholomorphic maps
%
%%%%%%%%%%%%%%%%%%%%%%%%%%%%%%%%%%%%%%%%%%%%%%%%%%%%%

\begin{theorem} % [Splitting Lemma]
\label{th:splitting}
Assume that $X$ is a complex space and $X'$ is a closed complex subvariety of $X$ containing its singular locus $X_\sing$. 
Let $(A,B)$ be a Cartan pair in $X$ such that $C:=A\cap B\subset X\setminus X'$. For any open set $\wt C \subset X$ 
containing $C$ there exist open sets $A'\supset A$, $B'\supset B$, $C'\supset C$ in $X$, with $C'\subset A'\cap B'\subset \wt C$,
 satisfying the following property. Given a number $\eta>0$, there exists a number $\epsilon_\eta >0$ such that for each 
 holomorphic map $\gamma\colon \wt C\to X$ with $\dist_{\wt C}(\gamma,\Id) < \epsilon_\eta$ there exist biholomorphic maps 
 $\alpha = \alpha_\gamma \colon A'\to \alpha(A') \subset X$ and $\beta = \beta_\gamma \colon B'\to \beta(B') \subset X$ 
 satisfying the following properties:
\begin{itemize}
\item[\rm (a)] $\gamma\circ \alpha = \beta$ on $C'$, 
\item[\rm (b)] $\dist_{A'}(\alpha,\Id) < \eta$ and $\dist_{B'}(\beta,\Id) < \eta$, and
\item[\rm (c)] $\alpha$ and $\beta$ are tangent to the identity map to any given finite order along the subvariety $X'$ 
intersected with their respective domains. 
\end{itemize}
\end{theorem}

In view of Lemma \ref{lem:closetoId} we can shrinking the set $\wt C$ if necessary and assume that $\gamma$ 
is biholomorphic onto its image. The crucial property (a) then furnishes a compositional splitting of $\gamma$. 
As in \cite{FF:Acta}, the proof will also show that the maps $\alpha_\gamma$ and $\beta_\gamma$ can be chosen to depend 
smoothly on $\gamma$ such that $\alpha_\Id=\Id$ and $\beta_\Id=\Id$.

The proof follows in spirit that of \cite[Theorem 4.1]{FF:Acta}, but is technically more involved. 
We embed a Stein neighbor\-hood of the Stein compact $D=A\cup B$ in $X$ as a closed complex subvariety of a 
complex Euclidean space $\C^N$. We then use a holomorphic retraction on a neighborhood $\Omega\subset \C^N$ of the 
Stein compact $C=A\cap B \subset X_\reg$ in order to transport the linearized splitting problem to a suitable 
1-parameter family of Cartan pairs in $\C^N$; see Lemma \ref{lem:CP}. 
(We have been unable to apply \cite[Theorem 4.1]{FF:Acta} directly since the resulting biholomorphic maps 
$\alpha,\beta$ need not map the subvariety $X$ to itself.)
From this point on we perform an iteration, similar to the one in \cite{FF:Acta}, 
in which the domains of maps shrink by a controlled amount at every step and the error term converges to zero quadratically.

\begin{proof}[Proof of Theorem \ref{th:splitting}]
Replacing $X$ by an open Stein neighborhood of $D$ we may assume that $X$ is a closed complex subvariety of a Euclidean 
space $\C^N$ \cite[Theorem 2.2.8]{FF:book}. The pair $(A,B)$ is then also a Cartan pair in $\C^N$ 
\cite[Lemma 5.7.2, p.\ 210]{FF:book}. We shall assume that the distance function dist on $X$ is induced by the 
Euclidean distance on $\C^N$.

By Cartan's Theorem $A$ there exist entire functions $h_1,\ldots, h_l\in\cO(\C^N)$ such that 
\[
	X=\{z\in \C^N\colon h_i(z)=0,\quad i=1,\ldots,l \}
\]
and $h_1,\ldots,h_l$ generate the ideal sheaf $\cJ_X$ of $X$. (We shall only need finite ideal generation on compact subsets 
of $\C^N$, but in our case this actually holds globally since $X$ is a relatively compact subset of the original Stein space.) 
Consider the analytic subsheaf $\cT_X \subset \cO^N_{\C^N}$ whose stalk $\cT_{X,p}$ at any point $p\in \C^N$ consists 
of all $N$-tuples $(g_1,\ldots,g_N) \in \cO_{\C^N,p}^N$ satisfying the system of equations
\[
     \sum_{j=1}^N g_j \frac{\di h_i}{\di z_j}\in \cJ_{X,p},\quad  i=1,\ldots,l.
\]
The condition is void when $p\notin X$, while at points $p\in X$ it means that the vector 
$V(p)=(g_1(p),\ldots,g_N(p)) \in \C^N \cong T_p\C^N$ is Zariski tangent to $X$. Observe that $\cT_X$ is the preimage 
of the coherent subsheaf $(\cJ_X)^l \subset \cO_{\C^N}^l$ under the homomorphism 
$\sigma\colon \cO^N_{\C^N} \to \cO_{\C^N}^l$ whose $i$-th component equals 
$\sigma_i(g_1,\ldots,g_N) = \sum_{j=1}^N g_j \frac{\di h_i}{\di z_j}$. Therefore $\cT_X$ is a coherent analytic 
subsheaf of $\cO^N_{\C^N}$. Sections of $\cT_X$ are holomorphic vector fields on $\C^N$ which are tangent to $X$
along $X$.  (Note that the quotient $\cT_X/\cJ_X \cT_X$, restricted to $X$, is the {\em tangent sheaf} of $X$ \cite{Fischer}.)

Denote by $\cJ_{X'}$ the sheaf of ideals of the subvariety $X'\subset X$. Fix an integer $n_0 \in\N$ and consider the 
coherent analytic sheaf $\cE:=\cJ^{n_0}_{X'}\cT_X$ on $\C^N$. By Cartan's Theorem A there exist sections $V_1,\ldots, V_m$ 
of $\cE$ that generate $\cE$ over the compact set $C=A\cap B \subset X\setminus X'$. These sections are holomorphic vector 
fields on $\C^N$ which are tangent to $X$ and vanish to order $n_0$ on the subvariety $X'$. Furthermore, as $C$ is contained 
in the regular locus $X_\reg$ of $X$ and $TX_\reg$ is the usual tangent bundle, the vectors $V_1(p),\ldots, V_m(p)\in T_p\C^N$ 
span the tangent space $T_p X\subset T_p\C^N$ at every  point $p\in X$ in a neighborhood of $C$. 

Denote by $\phi^j_t$ the local holomorphic flow of the vector field $V_j$ for a complex value of time $t$. 
For each point $z\in \C^N$ the flow $\phi^j_t(z)$ is defined for $t$ in a neighborhood of $0\in\C$. 
Let $t=(t_1,\ldots,t_m)$ be holomorphic coordinates on $\C^m$. The map
\begin{equation}
\label{eq:spray}
    s(z,t)= s(z,t_1,\ldots,t_m)=\phi^1_{t_1}\circ \cdots \circ \phi^m_{t_m}(z),\quad z\in \C^N,
\end{equation}
is defined and holomorphic on an open neighborhood of $\C^N \times \{0\}^m$ in $\C^N\times\C^m$ and assumes values in 
$\C^N$. Since the vector fields $V_j$ are tangent to $X$, we have $s(z,t)\in X$ for all $t$ whenever $z\in X$. For any 
point $z\in X$ we denote by 
\begin{equation}\label{eq:Vds}
	Vd(s)_z = \frac{\partial}{\partial t}\bigg|_{t=0} s(z,t) \colon \C^m \to T_z X 
\end{equation}
the partial differential of $s$ at $z$ in the fiber direction; we call $Vd(s)$ the {\em vertical derivative} of $s$ 
over the subvariety $X$. The definition of the flow of a vector field implies  
\[
	\frac{\partial s(z,t)}{\partial t_j}\Big|_{t=0}= V_j(z),\quad j=1,\ldots,m.
\]
Since the vectors $V_1(z),\ldots, V_m(z)$ span the tangent space $T_z X$ at every point $z \in C$, the vertical derivative 
$Vd(s)$ (\ref{eq:Vds}) is surjective over a neighborhood of $C$. Thus $s$ is a {\em local holomorphic spray} on $\C^N$, 
and the restriction of $s$ to a neighborhood of $C$ in $X$ is {\em dominating spray} (cf.\ \cite[p.\ 203]{FF:book} for these notions). 

Fix an open Stein set $U_0 \Subset X\setminus X' \subset X_\reg$ such that $C\subset U_0$ and $Vd(s)$ is surjective over 
$\overline U_0$. It follows that $U_0\times \C^m= E\oplus E'$, where $E'=\ker Vd(s)|_{U_0}$ and $E$ is some complementary 
to $E'$ holomorphic vector subbundle of $U_0\times\C^m$. (Such $E$ exists since $U_0$ is Stein and hence every holomorphic
vector subbundle splits the bundle.) Then $Vd(s)\colon E|_{U_0} \to TX|_{U_0}=TU_0$ is an isomorphism of holomorphic 
vector bundles. By the inverse mapping theorem, the restriction of $s$ to the fiber $E_z$ for any $z\in U_0$ maps an open
 neighborhood of the origin in $E_z$ biholomorphically onto an open neighborhood of the point $z$ in $X$.  Shrinking $U_0$ 
 slightly around $C$ we get an open set $U_1\supset C$ in $X$ such that the following holds (cf.\ \cite[Lemma 4.4]{FF:Acta}). 

%
%
%   LIFTING OF A MAP TO A SECTION
%
%
\begin{lemma}
\label{lem:lifting}
There are a neighborhood $U_1 \subset X\setminus X'$ of $C$ and constants $\epsilon_1>0$ and $M_1\ge 1$ such that for 
every open set $U\subset U_1$ and every holomorphic map $\gamma\colon U\to \gamma(U)\subset X$ satisfying 
$\dist_U(\gamma,\Id)<\epsilon_1$ there exists a unique holomorphic section $c\colon U\to E|_U$ satisfying
\[
	\gamma(z)=s(z,c(z))\ \ \forall z\in U, \quad 
	M_1^{-1} \,\dist_U(\gamma,\Id) \le ||c||_U \le M_1 \,\dist_U(\gamma,\Id).
\]
\end{lemma}

Since $E$ is a subbundle of the trivial bundle $U\times \C^m$, we may consider any section $c$ in Lemma \ref{lem:lifting} 
as a holomorphic map $U\to \C^m$, and $||c||_U$ denotes the sup-norm of $c$ on $U$ measured with respect to the 
Euclidean metric on $\C^m$. As $C=A\cap B\subset X_\reg$ is a Stein compact, the Docquier-Grauert theorem 
\cite{Docquier-Grauert} (see also \cite[Theorem 3.3.3, p.\ 67]{FF:book}) furnishes an open Stein neighborhood 
$\Omega \Subset \C^N$ of $C$ and a holomorphic retraction 
\begin{equation}\label{eq:rho}
	\rho\colon \Omega\to \Omega\cap X \Subset X_\reg.
\end{equation}
The map   
\begin{equation}\label{eq:T}
	T\colon \cO(\Omega\cap X) \to \cO(\Omega),\quad c\mapsto Tc=c\circ \rho,
\end{equation}
is then a bounded extension operator satisfying $||Tc||_\Omega = ||c||_{\Omega\cap X}$. By choosing $\Omega$ small enough 
we may assume that $\overline \Omega\cap X \subset \wt C$, where $\wt C$ is as in the statement of Theorem \ref{th:splitting}. 
We fix the domain $\Omega$, the retraction $\rho$, and the extension operator $T$ for the rest of the proof.  

The following lemma provides the key geometric ingredient in the proof of Theorem \ref{th:splitting}.

%
%
%  CONSTRUCTION OF A FAMILY OF CARTAN PAIRS
%
%
\begin{lemma}\label{lem:CP}
Assume that $X$ is a closed complex subvariety of $\C^N$. Let $(A,B)$ be a Cartan pair in $X$ such that 
$C:=A\cap B\subset X_\reg$. Let $\Omega \Subset \C^N$  be an open Stein neighborhood of $C$ and
$\rho\colon \Omega\to\Omega\cap X$ be a holomorphic retraction (\ref{eq:rho}).  
Let $U_A,U_B$ be open sets in $\C^N$ such that $A\subset U_A$, $B\subset U_B$, and $U_A\cap U_B \Subset \Omega$. 
Then there exists a family of smoothly bounded strongly pseudoconvex Cartan pairs $(A_t,B_t)$ in $\C^N$, depending 
smoothly on the parameter $t\in [0,t_0]$ for some $t_0>0$, satisfying the following properties:
\begin{itemize}
\item[\rm (i)]   For any pair of numbers $t,\tau$ such that $0\le t< \tau \le t_0$ we have 
$A \subset A_t \subset A_\tau \Subset U_A$, $B \subset B_t \subset B_\tau \Subset U_B$, and 
\[
		\biggl(\, \bigcup_{t\in [0,t_0]} \overline{A_t\setminus B_t}\biggr) \cap 
		\biggl(\, \bigcup_{t\in [0,t_0]} \overline{B_t\setminus A_t}\biggr) =\emptyset.
\]
\item[\rm (ii)]  Set $C_t:=A_t\cap B_t$. For any pair of numbers $t,\tau$ such that $0\le t< \tau \le t_0$ the distances 
$\dist(A_t,\C^N\setminus A_\tau)$, $\dist(B_t,\C^N\setminus B_\tau)$, and $\dist(C_t,\C^N\setminus C_\tau)$ are $\ge \tau-t>0$.
\item[\rm (iii)]  For every $t\in [0,t_0]$ the boundaries $bA_t$, $bB_t$, and $bC_t$ intersect $X$ transversely at any 
inter\-section point belonging to $\Omega\cap X$.
\item[\rm (iv)]  $\rho(C_t) = C_t \cap X$ for every $t\in[0,t_0]$. 
\end{itemize}
\end{lemma}

\begin{proof}
We shall modify the proof of Proposition 5.7.3 in \cite[p.\ 210]{FF:book} so as to also obtain property (iv) which will be crucial.  
(For a similar result see \cite[Proposition 4.4]{Stopar}.)

We shall use the function $\rmax\{x,y\}$ on $\R^2$, the {\em regularized maximum} of $x$ and $y$ 
(see e.g.\ \cite[p.\ 61]{FF:book}). It depends on a positive parameter which will be chosen as close to zero as needed at 
each application. We have $\max\{x,y\} \le \rmax\{x,y\}$, the two functions equal outside a small conical neighborhood 
of the diagonal $\{x=y\}$, and they can be made arbitrarily close by choosing the parameter small enough. 
The $\rmax$ of two smooth strongly plurisubharmonic functions is still smooth strongly plurisubharmonic. The 
domain $\{\rmax\{\phi,\psi\}<0\}$ is obtained by smoothing the corners of the intersection $\{\phi<0\}\cap \{\psi<0\}$.

Since  $C=A\cap B$ is a Stein compact, there is a smoothly bounded strongly pseudoconvex domain 
$V \Subset U_A\cap U_B$ such that $C\subset V$ and $bV$ intersects $X$ transversely. 
Let $\theta\colon V_0\to \R$ be a smooth strongly plurisubharmonic defining function for $V$, where  
$V_0 \subset U_A\cap U_B$ is an open neighborhood of $\overline{V}$.
Given a subset $A\subset \C^N$ and a number $r>0$, we set 
\[ 
		A(r) = \{z\in \C^N \colon |z-p| < r\ {\rm for\ some\ }  p\in A\}.
\]
It is elementary that $(A\cup B)(r) = A(r) \cup B(r)$ and $(A\cap B)(r)\subset A(r)\cap B(r)$. Since $A$ and $B$ are 
compact and separated,  we also have 
\[	
		(A\cap B)(r)=A(r)\cap B(r),\quad
		 \overline{A(r)\setminus B(r)} \cap \overline{B(r)\setminus A(r)} = \emptyset
\]
for all sufficiently small $r>0$  (cf.\ \cite[Lemma 5.7.4]{FF:book}).  By choosing $r>0$ small enough we 
can also ensure that 
\begin{equation}\label{eq:Cr}
	C(r)=A(r)\cap B(r) \Subset V.
\end{equation}
Fix such a number $r$. Since $A\cup B$ is a Stein compact, there is a smoothly bounded strongly pseudo\-convex 
Stein domain $\Omega_0 \subset \C^N$ such that  
\[
	A\cup B \subset \Omega_0 \Subset A(r)\cup B(r). 
\]
Pick a smooth strongly plurisubharmonic function $\phi \colon \Omega'_0\to\R$ on an 
open set $\Omega'_0\supset \overline \Omega_0$ such that $\Omega_0=\{z\in \Omega'_0\colon \phi(z)<0\}$ and $d\phi\ne 0$ 
on $b\Omega_0=\{\phi=0\}$. We may assume that $\Omega'_0\Subset A(r) \cup B(r)$. Choose a number  $\epsilon_0>0$ 
such that 
\begin{equation}
\label{eq:Omega1}
	\Omega'_1 := \{z\in \Omega'_0\colon \phi(z) < 3 \epsilon_0\} \Subset \Omega'_0.  
\end{equation}

By \cite[Lemma 2.2]{Richberg} % (or \cite[Proposition 4.1]{Stopar}) 
there exists a smooth function $\psi\ge 0$ on $\C^N$ such that $\{\psi=0\}=X$, $\psi$ is strongly plurisubharmonic on 
$\C^N\setminus X=\{\psi>0\}$, and the Levi form of $\psi$ at any point $z\in X$ is positive except on the 
tangent space $T_z X$. Choose a smooth function $\chi: \C^N\to [0,1]$ which equals $0$ on a neighborhood of 
$\overline{U}_A \cap \overline{U}_B \subset \Omega$ and equals $1$ on a neighborhood of $\C^N\setminus\Omega$.  
Since $\phi$ is strongly plurisubharmonic on $\Omega'_0$, there is a number $\epsilon \in (0,\epsilon_0)$ such that 
the functions $\phi-2\epsilon \chi$ and $\phi+\epsilon \chi$ are strongly plurisubharmonic on $\Omega'_1$; fix such 
$\epsilon$. Given constants $M,M'>0$ (to be determined later) we consider the functions
\begin{equation}
\label{eq:Phi}
	\Phi_1= (\phi - 2\epsilon \chi)\circ\rho + M\psi,\quad
	\Phi_2= \phi - 2\epsilon + \epsilon\chi +M'\psi, \quad
	\Phi = \rmax \bigl\{\Phi_1,\Phi_2\}.
\end{equation}
The function $\Phi_1$ is defined on $\Omega\cap \Omega'_0$ % (since $\rho$ is defined on $\Omega$), 
while $\Phi_2$ is defined on $\Omega'_0$. We shall see that for suitable choices of $M>0,M'>0$ the function $\Phi$ is 
well defined, smooth and strongly plurisubharmonic on the domain $\Omega'_1=\{\phi<3\epsilon_0\}$ (\ref{eq:Omega1}), 
and for all $t\in \R$ sufficiently close to $0$ we have 
\begin{equation}
\label{eq:Dt}
		A\cup B  \subset D_t:=\{z\in \Omega'_1 \colon \Phi(z) <t\} \Subset \Omega'_1.
\end{equation}

Since $\phi+\epsilon\chi$ is strongly plurisubharmonic on $\Omega'_1$ and $\psi$ is plurisubharmonic on $\C^N$, 
the function $\Phi_2$ is strongly pluri\-sub\-harmonic on $\Omega'_1$ for any choice of $M'>0$. 

Consider now $\Phi_1$. By the choice of $\epsilon$ the function $\phi-2\epsilon \chi$ is strongly plurisubharmonic on 
$\Omega'_1$, whence $\Phi_1$ is strongly plurisubharmonic on $\Omega\cap \Omega'_1$. Indeed, the first summand 
$(\phi- 2\epsilon \chi) \circ \rho$ is plurisubharmonic and its Levi form at any point $z\in \Omega\cap X$ is positive 
definite in the directions tangential to $X$. 
The second summand $M\psi$ is strongly plurisubhamonic on $\C^N \setminus X$, and its Levi form at points of $X$ is 
positive in directions that are not tangential to $X$. Hence the Levi form of the sum is positive everywhere. Observe that 
$\Phi_1>0$ near $b\Omega'_1 \cap \Omega \cap X$ by the definition of $\Omega'_1$ (\ref{eq:Omega1}) and the fact that 
$2\epsilon\chi \le 2\epsilon < 2\epsilon_0$. By choosing $M>0$ sufficiently big we can thus ensure that $\Phi_1>0$ 
near $b\Omega'_1 \cap \Omega$. We fix such $M$ for the rest of the proof.

Next we show that $\Phi = \rmax \bigl\{\Phi_1,\Phi_2\}$ is well defined if the constant $M'>0$ in $\Phi_2$ is chosen big enough. 
We need to ensure that $\Phi_1<\Phi_2$ on the domain of $\Phi_1$ near the boundary of $\Omega$, so $\Phi_2$ takes over in 
$\rmax$ before we run out of the domain of $\Phi_1$. On $X=\{\psi=0\}$ this is clear since near $b\Omega$ we have $\chi=1$ 
and hence $\Phi_1=\phi-2\epsilon < \phi-\epsilon=\Phi_2$. By choosing $M'>0$ sufficiently big we get $\Phi_1< \Phi_2$ on 
a neighborhood of $b\Omega \cap \Omega'_1$; hence $\rmax\{\Phi_1,\Phi_2\}$ is well defined on $\Omega'_1$. By increasing 
$M'$ we can also ensure that $\Phi>0$ near $b\Omega'_1$, so the domains $D_t=\{\Phi<t\}$ for $t$ close to $0$ 
(say $|t|\le t_1$ for some $t_1>0$) 
satisfy (\ref{eq:Dt}). By Sard's theorem and compactness of the level sets of $\Phi$ we can find a nontrivial interval 
$I \subset [-t_1,+t_1]$ which contains no critical values of $\Phi$ and of $\Phi|_{X\cap \overline \Omega}$. 
Hence $D_t$ for $t\in I$ are smoothly bounded strongly pseudoconvex domains intersecting $X$ transversely 
within $\Omega$.

On the intersection of the domain of $\Phi_1$ with the set $\{\chi=0\}$ (in particular, on $U_A\cap U_B \cap \Omega'_1$) 
we have $\Phi_1 = \phi\circ \rho+M\psi \ge \phi\circ\rho$, so on this set the retraction $\rho$ (\ref{eq:rho}) projects 
$\{\Phi_1 < t\}$ to $\{\Phi_1 <t\}  \cap X$. Furthermore, on $X\cap \{\chi=0\}$ we have 
$\Phi_1=\phi > \phi-2\epsilon=\Phi_2$. This shows that the domain $D_t=\{\Phi<t\}$
agrees with $\{\Phi_1 < t\}$ near $X\cap \{\chi=0\}$. It follows that  
\begin{equation}\label{eq:inclusion}
	\rho\left( D_t \cap \{\chi=0\}\right)  = D_t \cap X \cap \{\chi=0\}. 
\end{equation}

It remains to find a Cartan pair decomposition $(A_t,B_t)$ of $D_t$. 
Recall that $\overline{C(r)} \subset V=\{\theta<0\}$ by (\ref{eq:Cr}). 
Replacing $\theta$ by $c\theta$ for a suitably chosen constant $c>0$ we may therefore assume that 
$\theta<\Phi$ on $C(r)\cap \Omega'_1$. Perturbing $\theta$ and $V$ slightly 
we can ensure that the real hypersurfaces $bV$ and $bD_0=\{\Phi=0\}$ intersect transversely. 
The function 
\[
	\phi_C=\rmax\{\phi,\theta\} \colon \Omega'_1 \cap V_0 \to\R
\]
is smooth strongly plurisubharmonic. For every $t\in I$ the set 
\[
	 C'_t:= \{z\in \Omega'_1 \cap V_0 \cap X \colon \phi_C(z)< t\} \subset X_\reg
\]
is a smoothly bounded strongly pseudoconvex domain.
We have $C(r)\cap X \subset C'_t$ in view of (\ref{eq:Cr}) and $C'_t\subset D'_t:=D_t \cap X$ since $\phi_C\ge \phi$. 
As  $\theta<\phi$ on $\overline{C(r)} \cap \Omega'_1\, \cap X$, we have $\phi_C=\phi$ there,
so the boundaries $bC'_t$ and $bD'_t$ coincide along their intersection with the compact set $\overline{C(r)} \cap X$. 
Hence $C'_t$ separates $D'_t$ in the sense of a Cartan pair, i.e., we have $D'_t=A'_t\cup B'_t$ and $A'_t\cap B'_t=C'_t$.
Set  
\[
	\Theta =\rmax\bigl\{ \phi_C\circ \rho + M\psi, \Phi\bigr\},\quad C_t=\{\Theta<t\},
\]
where $\Phi=\rmax\{\Phi_1,\Phi_2\}$ (\ref{eq:Phi}) and $M$ is the constant in the definition of $\Phi_1$. 
One easily verifies that $C_t$ is a strongly pseudoconvex domain which separates $D_t$ into a 
Cartan pair $(A_t,B_t)$ with $D_t=A_t\cap B_t$ and $C_t=A_t\cap B_t$. 
It follows from (\ref{eq:inclusion})  that $C_t$ satisfies Lemma \ref{lem:CP}-(iv). 
By decreasing the parameter interval $I$ we ensure that $\Theta$ and $\Theta|_X$ have no critical values in $I$, 
so the boundaries $bC_t$ for $t\in I$ are smooth and intersect $X$ transversely. The same is then true for the domains 
$A_t$ and $B_t$ since their boundaries are contained in $bD_t\cup bC_t$. Reparametrizing the $t$ variable and changing 
the functions $\Phi$ and $\Theta$ by an additive constant we may assume that $I=[0,t_0]$ for some $t_0>0$ and 
property (ii) holds. The remaining properties of $A_t,B_t$ and $C_t$ follow directly from the construction. 
\end{proof}

Given an open set $U\subset X$ and a number $\delta>0$, we shall use the notation 
\[
	U(\delta) =\{z\in X\colon \dist(z,U)< \delta\}.
\]
Recall that $s$ is the spray (\ref{eq:spray}) and $M_1$ is the constant from Lemma \ref{lem:lifting}.

The following lemma is a special case of \cite[Lemma 4.5]{FF:Acta}. 

%
%
%   ACTA - LEMMA 4.5
%
%
\begin{lemma}\label{lem:ActaLemma4.5}
Let $U_1 \subset X\setminus X'$ be the open set from Lemma \ref{lem:lifting}. There exist constants $\delta_0>0$ (small) 
and $M_2>0$ (big) with the following property. Let $0< \delta <\delta_0$ and $0< 4\epsilon <\delta$. Let $U$ be an open 
set in $X$ such that $U(\delta)\subset U_1$. Assume that $\alpha,\beta,\gamma \colon U(\delta)\to X$ are holomorphic maps 
which are $\epsilon$-close to the identity on $U(\delta)$. Then $\wt \gamma:= \beta^{-1} \circ \gamma \circ \alpha \colon U\to X$ 
is a well defined holomorphic map. Write 
\begin{eqnarray*}
	\alpha(z) = s(z,a(z)), &\quad& \beta(z)=s(z,b(z)), \\
            \gamma(z)=s(z,c(z)), &\quad& \wt\gamma(x) = s(z, \wt c(z)), 
\end{eqnarray*}
where $a,b,c$ are holomorphic sections of the vector bundle $E|_{U(\delta)}$ and $\wt c$ is a holomorphic 
section of $E|_U$ furnished by Lemma \ref{lem:lifting}. If $c=b-a$ holds on $U(\delta)$, then 
\[
	||\wt c||_U\le M_2\delta^{-1} \epsilon^2 \quad \text{and}\quad 
	\dist_U(\wt \gamma,\Id) \le M_1M_2\delta^{-1} \epsilon^2.
\]
\end{lemma}

The next lemma provides a solution of the Cousin-I problem with bounds on the family of strongly pseudoconvex domains 
$D_t=A_t\cup B_t$ from Lemma \ref{lem:CP}, intersected with the subvariety $X$. We denote by $\cH^\infty(D)$ the 
Banach space of all bounded holomorphic functions on $D$.

%
%
%   LEMMA ON LINEAR SPLITTING
%
%
\begin{lemma}
\label{lem:Cousin}
Let $(A_t,B_t)$ $(t\in [0,t_0])$ be a family of strongly pseudoconvex Cartan pairs furnished by  Lemma \ref{lem:CP}. 
There is constant $M_3>0$ with the following property. 
For every $t\in [0,t_0]$ and $c\in \cH^\infty(C_t\cap X)$ there exist functions 
$a\in\cH^\infty(A_t)$ and $b\in \cH^\infty(B_t)$ such that
\[
	c=b-a \ \text{on}\  C_t \cap X, \quad
	||a||_{A_t} \le M_3 ||c||_{C_t\cap X}, \quad
	||b||_{B_t} \le M_3 ||c||_{C_t\cap X}. 
\] 
Functions $a$ and $b$ are given by bounded linear operators applied to $c$.
\end{lemma}

\begin{proof}
We begin by finding a constant $M_3>0$ independent of $t\in[0,t_0]$ and linear operators 
\[
	\cA_t \colon \cH^\infty(C_t)\to \cH^\infty(A_t),\quad 
	\cB_t \colon \cH^\infty(C_t)\to \cH^\infty(B_t)
\]
such that for every $g\in \cH^\infty(C_t)$ $(t\in [0,t_0])$ we have that
\begin{equation} \label{eq:diff}
	g = \cA_t(g)-\cB_t(g) \ \ \text{on} \ \ C_t
\end{equation}
and the estimates
\begin{equation}\label{eq:estimate}
		||\cA_t (g)||_{A_t} \le M_3 ||g||_{C_t}, \qquad ||\cB_t (g)||_{B_t} \le M_3 ||g||_{C_t}.
\end{equation}
The proof is similar to that of \cite[Lemma 5.8.2, p.\ 212]{FF:book} and uses standard techniques; we include it for completeness. 

In view of Lemma \ref{lem:CP}-(i) there is a smooth function $\xi\colon \C^N\to[0,1]$ such that $\xi=0$ on 
$\bigcup_{t\in [0,1]} \overline{A_t\setminus B_t}$ and $\xi=1$ on $\bigcup_{t\in [0,1]} \overline{B_t\setminus A_t}$. 
For any $g \in \cH^\infty(C_t)$ the product $\xi g$ extends to a bounded smooth function on the domain $A_t$ that vanishes on 
$A_t\setminus B_t$, and $(\xi-1)g$ extends to a bounded smooth function on $B_t$ that vanishes on $B_t\setminus A_t$. 
Furthermore, $\dibar(\xi g)= \dibar((\xi-1)g)=g\, \dibar \xi$ is a smooth bounded $(0,1)$-form on the strongly pseudoconvex 
domain $D_t=A_t\cup B_t$ with support contained in $C_t=A_t\cap B_t$. 

Let $S_t$ be a sup-norm bounded linear solution operator to the $\dibar$-equation on $D_t$ at the level of $(0,1)$-forms. 
(Such $S_t$ can be found as a Henkin-Ram\'irez integral kernel operator; see the monographs by Henkin and Leiterer 
\cite{HL:TF} or Lieb and Michel \cite{Lieb-Michel}. The operators $S_t$ can be chosen to depend smoothly on the 
parameter $t\in [0,1]$. For small perturbations of a given strongly pseudoconvex domain this is evident from the construction 
and is stated explicitly in the cited sources; for compact families of strongly pseudoconvex domains the result follows by 
applying a smooth partition of unity on the parameter space.) Given $g \in \cH^\infty(C_t)$, set
\[
    \cA_t (g) = \xi g - S_t(g\dibar \xi\bigr) \in\cH^\infty(A_t), \qquad 
    \cB_t (g) = (\xi-1) g	-  S_t(g \dibar \xi\bigr) \in\cH^\infty(B_t).
\]    
It is immediate that these operators satisfy the stated properties. 

By Lemma \ref{lem:CP}-(iv) the map $c \mapsto T(c)= c\circ \rho$ (\ref{eq:T}) induces a linear extension operator 
$\cO(C_t\cap X)\to \cO(C_t)$ satisfying $||Tc||_{C_t}=||c||_{C_t\cap X}$. The compositions
\[
	\cA_t\circ T \colon \cH^\infty(C_t\cap X)\to \cH^\infty(A_t), \qquad
	\cB_t\circ T \colon \cH^\infty(C_t\cap X)\to \cH^\infty(B_t)
\]
are then bounded linear operators satisfying the conclusion of Lemma \ref{lem:Cousin}.    
\end{proof}

\begin{lemma}\label{lem:ActaLemma4.7}
Let $(A_t,B_t)=(A(t),B(t))$ $(t\in [0,t_0])$ be strongly pseudoconvex Cartan pairs furnished by 
Lemma \ref{lem:CP}. Set $C(t)=A(t)\cap B(t)$. 
Let $\delta_0>0$ be chosen as in Lemma \ref{lem:ActaLemma4.5}.
Then there are constants $M_4,M_5 > 0$ satisfying the following property.
Let $0\le t <t+\delta \le t_0$ and $0< \delta <\delta_0$. 
For every holomorphic map $\gamma\colon C(t+\delta)\cap X \to X$ satisfying 
\[
	  \epsilon:= \dist_{C(t+\delta) \cap X}(\gamma,\Id) < \frac{\delta}{4M_4} 
\] 
there exist holomorphic maps $\alpha\colon A(t+\delta)\to \C^N$ and $\beta\colon B(t+\delta)\to \C^N$, 
tangent to the identity map to order $n_0$ along the subvariety $X'$ and satisfying the estimates
\begin{equation}\label{eq:est-alpha}
	\dist_{A(t+\delta)}(\alpha,\Id)< M_4\,\epsilon,\quad  
	\dist_{B(t+\delta)}(\beta,\Id) < M_4\,\epsilon,
\end{equation}
such that
\[
	\wt\gamma = \beta^{-1}\circ \gamma\circ\alpha \colon C(t)\cap X \to X
\]
is a well defined holomorphic map satisfying the estimate
\begin{equation}
\label{eq:Acta4.3}
	\dist_{C(t) \cap X}(\wt\gamma,\Id) < M_5\, \delta^{-1} \dist_{C(t+\delta)\cap X}(\gamma,\Id)^2
	= M_5\, \delta^{-1} \epsilon^2.   
\end{equation}
\end{lemma}

\begin{proof}
On the Banach space $\cH^\infty(D)^N$ we use the norm $||g||=\sum_{j=1}^N ||g_j||$, where $||g_j||$ is the sup-norm on $D$. 
By Lemma \ref{lem:lifting} there is a holomorphic section $c\colon C(t+\delta) \cap X \to E$ of the holomorphic vector bundle 
$E\to U_1$ such that $\gamma(z)=s(z,c(z))$ for $z\in C(t+\delta) \cap X$ and $||c||_{C(t+\delta) \cap X}\le M_1\epsilon$.  
(Here $M_1$ is the constant from Lemma \ref{lem:lifting}.) By Lemma \ref{lem:Cousin} we have $c=b-a$, where 
$a\in\cH^\infty(A_t)^N$ and $b\in \cH^\infty(B_t)^N$ satisfy the estimates
\[
		||a||_{A_t} \le N M_1 M_3 \epsilon, \qquad ||b||_{B_t} \le N M_1 M_3 \epsilon. 
\]
Let $s$ be the spray (\ref{eq:spray}). Set
\begin{eqnarray*}	
	\alpha(z) &=& s(z,a(z)), \qquad z\in A(t+\delta), \\
	\beta(z)  &=& s(z,b(z)), \qquad z\in B(t+\delta). 
\end{eqnarray*}	
The maps $\alpha$ and $\beta$ are tangent to the identity to order $n_0$ along the subvariety $X'$ and satisfy 
$\alpha(A_t \cap X)\subset X$ and $\beta(B_t\cap X)\subset X$. By Lemma \ref{lem:lifting} we have 
\[
	\dist_{A(t+\delta)}(\alpha,\Id)< N M_1^2 M_3\,\epsilon,\quad  
	\dist_{B(t+\delta)}(\beta,\Id) < N M_1^2M_3\,\epsilon. 
\] 
Setting $M_4=N M_1^2M_3$ we get the estimates (\ref{eq:est-alpha}). If the number 
$\epsilon=\dist_{C(t+\delta)}(\gamma,\Id)$ satisfies  $4M_4\epsilon<\delta$,
then by Lemma \ref{lem:ActaLemma4.5} the composition 
$\wt \gamma=\beta^{-1}\circ \gamma\circ \alpha$ is a well defined holomorphic map on $C(t)\cap X$ 
satisfying the estimate (\ref{eq:Acta4.3}) with the constant $M_5= M_2M_4^2$.
\end{proof}

We now complete the proof of Theorem \ref{th:splitting}.% as in \cite[\S 4]{FF:Acta}. This is 
by a recursive process, using Lemma \ref{lem:ActaLemma4.7} at every step.
The initial map $\gamma$ is defined on the set $\wt C\supset C(t_0)\cap X$. For each $k\in\Z_+$ we set  
\[
	t_k = t_0\prod_{j=1}^k (1-2^{-j})\quad \text{and} \quad  \delta_k = t_k - t_{k+1} = t_k 2^{-k-1}.  
\]
The sequence $t_k>0$ is decreasing, $t^* = \lim_{k\to \infty} t_k > 0$, $\delta_k > t^* 2^{-k-1}$ for all $k$, and 
$\sum_{k=0}^\infty \delta_k= t_0-t^*$. Set $A_k=A(t_k)$, $B_k=B(t_k)$, $C_k=C(t_k)=A_k\cap B_k$ and observe that
\[
	\bigcap_{k=0}^\infty A_k = \overline{A(t^*)},\qquad 
	\bigcap_{k=0}^\infty B_k = \overline{B(t^*)},\qquad 
	\bigcap_{k=0}^\infty C_k = \overline{C(t^*)}.
\]

To begin the induction, pick a number $\epsilon_0>0$ such that
$4M_4 \epsilon_0 <\delta_0=t_0/2$. Set $\gamma_0=\gamma$ and  assume that 
$\dist_{C_0\cap X}(\gamma_0,\Id)\le \epsilon_0$. Lemma \ref{lem:ActaLemma4.7} furnishes 
holomorphic maps $\alpha_1\colon A_1\to \C^N$ and $\beta_1\colon B_1\to \C^N$,  satisfying the estimates 
\[
	\dist_{A_{1}}(\alpha_{1},\Id)< M_4\,\epsilon_0, \qquad  
	\dist_{B_{1}}(\beta_{1},\Id) < M_4\,\epsilon_0
\]
(see (\ref{eq:est-alpha})), such that
$ 
	\gamma_1=\beta_1^{-1} \circ \gamma_0 \circ \alpha_1 \colon C_1\cap X \to X
$  
is a well defined holomorphic map satisfying the following estimate (cf.\ (\ref{eq:Acta4.3})):
\[
	\epsilon_1:= \dist_{C_1\cap X}(\gamma_1,\Id) < M_5 \delta_0^{-1} \epsilon^2_0  < 2M\epsilon_0^2
\]
where the constant $M>0$ is given by
\[
	M=M_5/t^*.
\] 
Assuming that $4M_4\epsilon_1 <\delta_1$ (which holds if $\epsilon_0>0$ is small enough),  
Lemma \ref{lem:ActaLemma4.7} furnishes holomorphic maps 
$\alpha_2\colon A_2\to \C^N$ and $\beta_2\colon B_2\to \C^N$ satisfying the estimates 
\[
	\dist_{A_2}(\alpha_{2},\Id)< M_4\,\epsilon_1, \qquad   \dist_{B_{2}}(\beta_{2},\Id) < M_4\,\epsilon_1
\]
and such that 
$
	\gamma_2=\beta_2^{-1}\circ \gamma_1 \circ \alpha_2 \colon C_2\cap X\to X
$ 
is a well defined holomorphic map satisfying the estimate
\[
	\epsilon_2:=\dist_{C_2\cap X}(\gamma_2,\Id)  < M_5 \delta_1^{-1} \epsilon^2_1  < 2^2 M \epsilon_1^2.
\]
Proceeding inductively, we obtain sequences of holomorphic maps
\[
	\alpha_k\colon A_k \to\C^N, \quad \beta_k \colon B_k\to \C^N, 
	\quad \gamma_k\colon C_k\cap X \to X
\] 
such that the following conditions hold for every $k=0,1,2,\ldots$:
\[
	\gamma_{k+1}=\beta_{k+1}^{-1}\circ \gamma_k\circ \alpha_{k+1} \colon C_{k+1}\cap X \to X,
\] 
\begin{equation}\label{eq:est-alphabis}
	\dist_{A_{k+1}}(\alpha_{k+1},\Id)< M_4\,\epsilon_k, \quad  
	\dist_{B_{k+1}}(\beta_{k+1},\Id) < M_4\,\epsilon_k,
\end{equation} 
\begin{equation}
\label{eq:Acta4.4}
	\epsilon_{k+1} := \dist_{C_{k+1}\cap X}(\gamma_{k+1},\Id) 
	< M_5 \delta_k^{-1} \epsilon^2_k  < 2^{k+1} M \epsilon_k^2.       
\end{equation}
The necessary condition for the induction to proceed is that $4M_4\epsilon_k<\delta_k$ holds for each $k=0,1,2,\ldots$. 
By \cite[Lemma 4.8]{FF:Acta} this condition is fulfilled as long as the initial number $\epsilon_0>0$ is chosen small enough,
and the resulting sequence $\epsilon_k>0$, defined recursively by (\ref{eq:Acta4.4}), then converges to zero very rapidly. 
In particular,  we can ensure that 
\begin{equation}\label{eq:eta}
	 M_4 \sum_{j=0}^{\infty} \epsilon_j  <\eta
\end{equation}
where $\eta>0$ is as in the statement of the theorem. 

The estimates (\ref{eq:est-alphabis}) imply that the compositions
\begin{equation}\label{eq:tilde-alpha}
	\wt\alpha_k=\alpha_1\circ \alpha_2\circ \cdots\circ \alpha_k \colon A_k\to \C^N,
	\quad
  	\wt \beta_k=\beta_1 \circ \beta_2 \circ \cdots\circ \beta_k \colon B_k\to \C^N
\end{equation}
are well defined holomorphic maps for $k=1,2,\ldots$  satisfying
\begin{equation}\label{eq:compose-k}
	\wt \beta_{k}\circ  \gamma_{k} = \gamma \circ \wt \alpha_{k}\quad \text{on}  
	\quad C_{k}\cap X.
\end{equation}

As $k\to\infty$, the estimate (\ref{eq:Acta4.4}) and  (\ref{eq:eta}) show that the sequence $\gamma_k$ 
converges to the identity map uniformly on $C(t^*)\cap X$.

Consider now the sequences $\wt\alpha_k$ and $\wt\beta_k$. 
From  (\ref{eq:est-alphabis}) and (\ref{eq:eta})  we clearly get that
\begin{equation}\label{eq:est-tilde-alpha}
	\dist_{A_{k}}(\wt\alpha_k,\Id) < M_4 \sum_{j=0}^{k-1} \epsilon_j <\eta,
	\quad 
	\dist_{B_{k}}(\wt\beta_k,\Id) < M_4 \sum_{j=0}^{k-1} \epsilon_j <\eta.
\end{equation}
Hence $\wt\alpha_k$ and $\wt\beta_k$ are normal families on $A(t^*)$ and $B(t^*)$, respectively.
Passing to convergent subsequences we get holomorphic maps 
$\alpha\colon A(t^*)\to \C^N$ and $\beta\colon B(t^*)\to \C^N$  satisfying
\[
	\dist_{A(t^*)}(\alpha,\Id) < \eta, 	\quad \dist_{B(t^*)}(\beta,\Id) < \eta, \quad 
	\gamma\circ \alpha=\beta \ \text{on}\ C(t^*)\cap X.
\]
The last equation follows from (\ref{eq:compose-k}).  Assuming that the numbers $\epsilon_0>0$ and $\eta>0$  
are  small enough, the maps $\alpha$, $\beta$ and $\gamma$ are biholomorphic on slightly smaller domains in 
view of Lemma \ref{lem:closetoId}.  Finally, since all maps $\alpha_k$ and $\beta_k$ in the sequence are tangent 
to the identity to order $n_0$ along the subvariety $X'$, the same is true for the maps $\alpha$ and $\beta$.
\end{proof}

\begin{remark}
1. The sequences $\wt\alpha_k$ and $\wt\beta_k$ in (\ref{eq:tilde-alpha}) actually converge
uniformly on any compact subset $K_A\subset A(t^*)$ and $K_B\subset B(t^*)$, respectively. 
Indeed, choose open domains $A'$ and $B'$ in $\C^N$ 
such that $K_A\subset A'\Subset A(t^*)$ and $K_B\subset B'\Subset B(t^*)$.
From  (\ref{eq:est-tilde-alpha}) and the Cauchy estimates we infer that these sequence are 
uniformly Lipschitz on  $A'$ and $B'$, respectively. 
From this and  (\ref{eq:est-tilde-alpha}) we easily see that the sequences are uniformly Cauchy, 
and hence uniformly convergent, on $K_A$ and $K_B$, respectively.

2. If $X$ is a Stein manifold with the {\em density property} in the sense of Varolin \cite[\S 4.10]{FF:book} and $(A,B)$ is a 
Cartan pair in $X$ such that the set $C=A\cap B$ is $\cO(X)$-convex, then the conclusion of Theorem \ref{th:splitting} 
holds for any biholomorphic map $\gamma$ on a neighborhood $U\subset X$ of $C$ which is isotopic to the identity map on $C$ 
through a smooth 1-parameter family of biholomorphic maps $\gamma_t\colon U\to \gamma_t(U)\subset X$ $(t\in [0,1])$ such that 
$\gamma_0=\Id$, $\gamma_1=\gamma$, and $\gamma_t(C)$ is $\cO(X)$-convex for every $t\in [0,1]$. The main result of the 
Anders\'en-Lempert theory (cf.\ \cite[Theorem 4.9.2]{FF:book} for $X=\C^n$ and \cite[Theorem 4.10.6]{FF:book} for the 
general case) implies that $\gamma$ can be approximated uniformly on a neighborhood of $C$ by holomorphic automorphisms 
$\phi\in\Aut(X)$. This allows us to write $\gamma=\phi\circ\tilde \gamma$,  where $\tilde\gamma$ is a biholomorphic map 
close to the identity on a neighborhood of $C$. Applying Theorem \ref{th:splitting} gives 
$\tilde \gamma=\tilde\beta\circ\alpha^{-1}$, where $\alpha$ and $\tilde \beta$ are biholomorphic maps close to the 
identity near  $A$ and $B$, respectively. Setting $\beta=\phi\circ\tilde \beta$ gives $\gamma=\beta\circ\alpha^{-1}$.
\qed \end{remark}

\begin{remark}\label{rem:3.2gen}
Theorem \ref{th:splitting} and its proof are amenable to various generalizations. In particular,
since all steps in the proof are obtained by using (possibly nonlinear) operators on various function spaces,
it immediately generalizes  to the parametric case when the data (in particular, the map $\gamma$ to be decomposed
in the form $\gamma=\beta\circ\alpha^{-1}$) depend on parameters. A more ambitious generalization would
amount to also letting the domains of these maps depend on parameters; this may be applicable in various
constructions. A first step in this direction has already been made in \cite{DFW}.
\qed \end{remark}

%%%%%%%%%%%%%%%%%%%%%%%%%%%%%%%%%%%%%%%%%%%%%%%%%%%%%%%
%																				%	
%																				%
%  Functions without critical points in the open strata                  							        		%
%																				%	
%																				%
%%%%%%%%%%%%%%%%%%%%%%%%%%%%%%%%%%%%%%%%%%%%%%%%%%%%%%%

\section{Functions without critical points in the top dimensional strata}  \label{sec:open-stratum}

In this section we construct holomorphic functions that have no critical points in the regular locus of a Stein space. 
The following result is the main inductive step in the proof of Theorem \ref{th:stratified} given in the following section.

\begin{theorem}
\label{th:th1}
Assume that $X$ is a Stein space, $X'\subset X$ is a closed complex subvariety of $X$ containing $X_\sing$, 
$P=\{p_1,p_2,\ldots\}$ is a closed discrete set in $X'$, $K$ is a compact $\cO(X)$-convex set in $X$ (possibly empty), 
and $f$ is a holomorphic function on a neighborhood of $K\cup X'$ such that $\Crit(f|_{U\setminus X'}) = \emptyset$ 
for some neighborhood $U\subset X$ of $K$. Then for any $\epsilon>0$ and integers $r\in\N$ and $n_k \in\N$ $(k=1,2,\ldots)$ 
there exists a holomorphic function $F\in\cO(X)$ satisfying the following conditions: 
\begin{itemize}
\item[\rm (i)]   $F-f$ vanishes to order $r$ along the subvariety $X'$, 
\item[\rm (ii)]  $F-f$ vanishes to order $n_k$ at the point $p_k\in P$ for every $k=1,2,\ldots$, 
\item[\rm (iii)] $||F-f||_K <\epsilon$, and 
\item[\rm (iv)]  $F$ has no critical points in $X\setminus X'$.
\end{itemize}
\end{theorem}

Applying Theorem \ref{th:th1} with the subvariety $X'=X_\sing$ we find holomorphic functions on $X$ that have no 
critical points in the smooth part $X_\reg$. 

\begin{remark}\label{rem:th1}
Theorem \ref{th:th1} implies at no cost the following result. Let $A=\{a_j\}$ be a closed discrete set in $X$ contained in 
$X\setminus (K\cup X')$. Then there exists a function $F\in \cO(X)$ satisfying conditions (i)--(iii) and also the condition
\begin{itemize}
\item[\rm (iv')] $\Crit(F|_{X\setminus X'}) = A$.
\end{itemize}
Indeed, choose any germs $g_{j}\in \cO_{X,a_j}$ at points $a_j\in A$ and apply Theorem \ref{th:th1} with the subvariety 
$X'_0=A \cup X'$, the discrete set $P_0=A\cup P$, and the function $f$ extended by $g_j$ to a small neighborhood of 
the point $a_j\in A$.  
\qed \end{remark}

We begin with a couple lemmas. The first one shows that we can replace the function $f$ in Theorem \ref{th:th1} by a 
function in $\cO(X)$.

\begin{lemma}
\label{lem:extension}
{\rm (Assumptions as in Theorem \ref{th:th1}.)} Let $L$ be a compact $\cO(L)$-convex set in $X$ such that $K\subset \mathring L$. 
Then there exists a function $\tilde f\in \cO(X)$ with the following properties:
\begin{itemize}
\item[\rm (a)]  $\tilde f -f$ vanishes to order $r$ along the subvariety $X'$, 
\item[\rm (b)]  $\tilde f-f$ vanishes to order $n_k$ at the point $p_k\in P$ for every $k=1,2,\ldots$, 
\item[\rm (c)]  $||\tilde f - f||_K <\epsilon$, and 
\item[\rm (d)]  there is a neighborhood $W\subset X$ of the compact set $K\cup (L\cap X')$ such that $\tilde f$ has no 
critical points in the set $W\setminus X'$.
\end{itemize}
\end{lemma}

\begin{proof} 
Let $\cE\subset \cO_X$ be the coherent sheaf of ideals whose stalk at any point $p_k\in P$ equals $\mgot_{p_k}^{n_k}$ and 
$\cE_x=\cO_{X,x}$ for every $x\in X\setminus P$. The product 
\begin{equation}
\label{eq:sheafE}
	\wt \cE:=\cE \cJ_{X'}^{r}  \subset \cO_X
\end{equation}
of $\cE$ and the $r$-th power of the ideal sheaf $\cJ_{X'}$ is a coherent sheaf of ideals in $\cO_X$. Consider the short 
exact sequence of sheaf homomorphisms
\[
	0\longrightarrow \wt \cE \longrightarrow \cO_X \longrightarrow \cO_X/\wt\cE \longrightarrow 0.
\]
Since the quotient sheaf $\cO_X/\wt \cE$ is supported on $X'$, the function $f$ determines a section of $\cO_X/\wt \cE$. 
Since $H^1(X;\wt\cE)=0$ by Theorem B, the same section is induced by a function $g\in\cO(X)$. Clearly $g$ satisfies conditions 
(a) and (b) of the lemma (with $g$ in place of $\tilde f$). 
To get condition (c) we proceed as follows. Cartan's Theorem A furnishes sections $\xi_1,\ldots,\xi_m \in\Gamma(X, \wt \cE)$ 
which generate the sheaf $\wt \cE$ over the compact set $K$. By the choice of $g$ the difference $f-g$ is a section of $\wt \cE$ 
over a neighborhood of $K$. Applying Theorem B to the epimorphism of coherent analytic sheaves 
$\cO_X^m \longrightarrow \wt\cE\longrightarrow 0$,  $(h_1,\ldots,h_m)\mapsto \sum_{i=1}^m h_i\xi_i$, 
we obtain $f=g + \sum_{i=1}^m h_i\xi_i$ on a neighborhood of $K$ for some holomorphic functions $h_i\in \cO(K)$. 
By the Oka-Weil theorem we can approximate the $h_i$'s uniformly on $K$ by functions $\tilde h_i\in \cO(X)$. 
The function $\tilde f = g+\sum_{i=1}^m \tilde h_i\xi_i \in\cO(X)$ satisfies properties (a)--(c). 
By the Stability Lemma \ref{lem:stability} and the Genericity Lemma \ref{lem:generic} we can also satisfy condition (d) 
by choosing $\tilde f$ generic and taking $\epsilon>0$ small enough.      
\end{proof}

The next lemma is the main step in the proof of Theorem \ref{th:th1};  here we
use the splitting lemma furnished by Theorem \ref{th:splitting}. 
Another key ingredient is the Runge approximation theorem 
for noncritical holomorphic functions on $\C^n$, furnished by Theorem 3.1 in \cite{FF:Acta}.

\begin{lemma} \label{lem:th1}
{\rm (Assumptions as in Theorem \ref{th:th1}.)} 
Let $L\subset X$ be a compact $\cO(X)$-convex set such that $K\subset \mathring L$. Then there exists a holomorphic 
function $F\in \cO(X)$ which satisfies conditions (i)--(iii) in Theorem \ref{th:th1} and also the following condition:
\begin{itemize}
\item[\rm (iv')] $\Crit(F|_{U'\setminus X'})  = \emptyset$, where $U'\subset X$ is an open neighborhood of $L$.
\end{itemize}
\end{lemma}

\begin{proof}
To simplify the exposition, we replace the number $r$ by the maximum of $r$ and the numbers $n_k\in\N$ over all points 
$p_k\in P\cap L$ (a finite set). Choosing $F$ to satisfy condition (i) for this new $r$, it will also satisfy condition (ii) at the points 
$p_k\in P\cap L$.

Let $W$ be the set from Lemma \ref{lem:extension}-(d). By \cite[Lemma 8.4, p.\ 662]{FP2} there exist finitely many 
compact $\cO(X)$-convex sets $A_0\subset A_1\subset \cdots \subset A_m=L$ such that 
$K \cup (L\cap X') \subset A_0 \subset W$ and for every $j=0,1,\ldots, m-1$ we have $A_{j+1}=A_j\cup B_j$, 
where $(A_j,B_j)$ is a Cartan pair (cf.\ Definition \ref{def:CP}) and $B_j\subset L \setminus X' \subset X_\reg$. 
Furthermore, the construction in \cite{FP3} gives for every $j=0,1,\ldots, m-1$ an open set $U_j\subset X_\reg$ containing 
$B_j$ and a biholomorphic map $\phi_j\colon U_j\to U'_j \subset \C^n$ onto an open subset of $\C^n$ 
(where $n$ is the dimension of $X$ at the points of $B_j$) such that the set $\phi_j(C_j)$ is polynomially convex in $\C^n$. 
(Here $C_j=A_j\cap B_j$.)  The proof of Lemma 8.4 in \cite{FP2} is written in the case when $X$ is a Stein manifold, 
but it also applies in the present situation, for example,  by embedding a relatively compact neighborhood of $L\subset X$ 
as a closed complex subvariety in a Euclidean space $\C^N$.

We first find a function $\wt F$ that is holomorphic on a neighborhood of $L$ and satisfies the conclusion of the lemma there. 
This is accomplished by a finite induction, starting with $F_0=f$ which by the assumption satisfies these properties on the 
open set $W\supset A_0$. We provide an outline and refer to \cite{FF:Acta} for further details.

By the assumption $F_0$ is noncritical on a neighborhood of the set $C_0=A_0\cap B_0$. Since $C_0$ is polynomially convex 
in a certain holomorphic coordinate system on a neighborhood of $B_0$ in $X\setminus X'\subset X_\reg$, 
Theorem 3.1 in \cite[p.\ 154]{FF:Acta} furnishes a noncritical holomorphic function $G_0$ on a neighborhood of $B_0$ in 
$X\setminus X'$ such that $G_0$ approximates $F_0$ as closely as desired uniformly on a neighborhood of $C_0$. 
Assuming that the approximation is close enough, we can apply Theorem \ref{th:splitting} to glue $F_0$ and $G_0$ 
into a new function $F_1$ that is holomorphic on a neighborhood of $A_0\cup B_0=A_1$ and has no critical points, 
except perhaps on subvariety $X'$. 
The gluing of $F_0$ and $G_0$ is accomplished by first finding a biholomorphic map $\gamma$ close to the 
identity on a neighborhood of the attaching set $C_0$ such that 
\[
	F_0=G_0\circ \gamma  \quad \text{on a neighborhood of}\ C_0. 
\]
Since $C_0$ is a Stein compact in the complex manifold $X\setminus X'$, such $\gamma$ is furnished by Lemma 5.1 
in \cite[p.\ 167]{FF:Acta}. If $\gamma$ is close enough to $\Id$ (which holds if $G_0$ is chosen sufficiently uniformly 
close to $F_0$ on a neighborhood of $C_0$), then Theorem \ref{th:splitting} furnishes a decomposition 
\[
	\gamma \circ\alpha=\beta,
\]
where $\alpha$ is a biholomorphic map close to the identity on a neighborhood of $A_0$ in $X$ and $\beta$ is a map 
with the analogous properties on a neighborhood of $B_0$ in $X$. By Theorem \ref{th:splitting} we can ensure in addition 
that $\alpha$ is tangent to the identity to order $r$ along the subvariety $X'$ intersected with its domain. 
(The domain of $\beta$ does not intersect $X'$.) Then 
\[
	F_0\circ \alpha = G_0\circ\beta \quad \text{on a neighborhood of}\ C_0,
\]
so the two sides amalgamate into a holomorphic function $F_1$ on a neighborhood of $A_0\cup B_0=A_1$. 
By the construction, $F_1$ approximates $F_0$ on a neighborhood of $A_0$, $F_1-F_0$ vanishes to order $r$ along $X'$, 
and $F_1$ is noncritical except perhaps on $X'$. The last property holds because the maps $\alpha$ and $\beta$ are 
biholomorphic on their respective domains and $\alpha|_{X'}$ is the identity. 

Repeating the same construction with $F_1$ we get the next function $F_2$ on a neighborhood of $A_2$, etc. 
In $m$ steps of this kind we find a function $\wt F=F_m$ on a neighborhood of the set $A_m=L$ satisfying the stated properties.

It remains to replace $\wt F$ by a function $F\in \cO(X)$ satisfying the same properties. This is done as in 
Lemma \ref{lem:extension} above. Let $\wt \cE$ be the sheaf (\ref{eq:sheafE}). Pick sections 
$\xi_1,\ldots,\xi_m\in\Gamma(X,\wt \cE)$ which generate $\wt\cE$ over the compact set $L$. 
By the construction of $\wt F$, the difference $\wt F-f$ is a section of $\wt\cE$ over a neighborhood of $L$. 
Hence Cartan's Theorem B furnishes holomorphic functions $\tilde h_1,\ldots, \tilde h_m\in \cO(U')$ an open set 
$U'\supset L$ such that
\[
		\wt F=f+\sum_{i=1}^m \tilde h_i\,\xi_i \quad \text{on}\ U'.
\]
Choose a compact $\cO(X)$-convex set $L'$ such that $L\subset \mathring L'\subset L'\subset U'$. Approximating each 
$\tilde h_i$ uniformly on $L'$ by a function $h_i\in \cO(X)$ and setting 
\[
		F=f+ \sum_{i=1}^m h_i \,\xi_i \in\cO(X)
\]
we get a function $F$ satisfying properties (i)--(iii). By Lemma \ref{lem:stability} the function $F$ also satisfies property (iv') 
provided that the differences $||h_i-\tilde h_i||_{L'}$ for $i=1,\ldots,m$ are small enough.
\end{proof}

\begin{proof}[Proof of Theorem \ref{th:th1}]
In view of Lemma \ref{lem:extension} we may assume that $f\in\cO(X)$. Choose an increasing sequence 
$K=K_0\subset K_1\subset \cdots \subset \bigcup_{i=0}^\infty K_i =X$ of compact $\cO(X)$-convex sets satisfying 
$K_i\subset \mathring K_{i+1}$ for every $i=0,1,\ldots$. Set $F_0=f$, $\epsilon_0=\epsilon/2$, and $r_0=r$. 
We inductively construct a sequence of functions $F_i\in\cO(X)$ and numbers $\epsilon_i>0$, $r_i\in \N$ such that 
the following conditions hold for every $i=0,1,2,\dots$:
\begin{itemize}
\item[\rm (a)] $\Crit(F_i|_{U_i\setminus X'})=\emptyset$ for an open neighborhood $U_i\supset K_i$,
\item[\rm (b)] $||F_{i} -F_{i-1}||_{K_{i-1}}<\epsilon_{i-1}$,
\item[\rm (c)] $F_{i}-F_{i-1}$ vanishes to order $r_{i-1}$ along the subvariety $X'$,
\item[\rm (d)] $F_{i}-F_{i-1}$ vanishes to order $n_k$ at each of the point $p_k\in P$, 
\item[\rm (e)] $0<\epsilon_{i} < \epsilon_{i-1}/2$ and $r_{i}\ge r_{i-1}$, and
\item[\rm (f)] if $F\in \cO(X)$ is such that $||F -F_{i}||_{K_{i}}< 2\epsilon_i$ and $F-F_{i}$ vanishes to order $r_{i}$ 
along $X'$, then $\Crit(F|_{U\setminus X'})=\emptyset$ for an open neighborhood $U\supset K_{i-1}$.
\end{itemize}
Assume that we have already found these quantities up to index $i-1$ for some $i \in\N$. (For $i=0$ the function $F_0$ 
satisfies condition (a) and the remaining conditions are void.) Lemma \ref{lem:th1} furnishes the next map $F_{i}\in\cO(X)$ 
which satisfies conditions (a)--(d). For this $F_{i}$ we then pick the next pair of numbers $\epsilon_{i}>0$ and $n_{i}\in\N$ 
such that conditions (e) and (f) hold. In view of the Stability Lemma \ref{lem:stability}, condition (f) holds by  as soon as 
$\epsilon_{i}>0$ is chosen small enough and $r_{i}\in \N$ is chosen big enough. This completes the induction step. 

It is straightforward to verify that the sequence $F_i$ converges uniformly on compacts in $X$ and the limit function 
$F= \lim_{i\to \infty} F_i \in \cO(X)$ satisfies the conclusion of Theorem \ref{th:th1}.
\end{proof}

In the proof of Theorem \ref{th:stratified} (see \S \ref{sec:stratified}) we shall combine Theorem \ref{th:th1} with the 
following lemma which provides extension from a subvariety and jet interpolation on a discrete set.

%
%   Extension with jet interpolation
%
\begin{lemma}[{\bf Extension with jet interpolation}]  \label{lem:interpolate}
Let $X$ be a Stein space, $X'$ a closed complex subvariety of $X$, and $P=\{p_1,p_2,\ldots\}$ a closed discrete subset of 
$X'$. Given a function $f\in \cO(X')$ and germs $f_k\in\cO_{X,p_k}$ for each $p_k\in P$ such that 
$f_{p_k} - (f_k|_{X'})_{p_k} \in \mgot_{X',p_k}^{n_k}$ for some $n_k\in \N$, there 
exists $F\in\cO(X)$ such that $F|_{X'}=f$ and $F_{p_k} -f_k\in \mgot_{X,p_k}^{n_k}$ for every $p_k\in P$. 
\end{lemma}

\begin{proof}
Let $\cJ_{X'}$ denote the sheaf of ideals of the subvariety $X'$. By Lemma \ref{lem:germs} there exists for every point 
$p_k\in P$ a germ $g_k\in\cO_{X,p_k}$ such that $f_k-g_k \in  \mgot_{X,p_k}^{n_k}$ and 
$(g_k|_{X'})_{p_k}=f_{p_k}  \in\cO_{X',p_k}$.  Pick a function $\tilde f\in \cO(X)$ with $\tilde f|_{X'}=f$; 
then $\tilde f_{p_k} - g_k \in \cJ_{X',p_k}$.  Let $\cE\subset \cO_X$ be the coherent sheaf of ideals whose stalk at 
any point $p_k\in P$ equals $\mgot_{p_k}^{n_k}$ and $\cE_x=\cO_{X,x}$ for every $x\in X\setminus P$. 
Consider the following short exact sequence of coherent analytic sheaves on $X$:
\[
	0\longrightarrow \cE \cJ_{X'} \longrightarrow \cJ_{X'} 
	\longrightarrow \cJ_{X'}/(\cE \cJ_{X'}) \longrightarrow 0.
\]
The quotient sheaf $\cJ_{X'}/(\cE \cJ_{X'})$ is supported on the discrete set $P$, and hence the collection of germs 
$\tilde f_{p_k} - g_k \in \cJ_{X',p_k}$ determines a section of this sheaf. Since $H^1(X;\cE \cJ_{X'})=0$ by  Theorem B, 
this section lifts to a section $h$ of $\cJ_{X'}$. 

Consider the function $F:=\tilde f-h \in\cO(X)$. We have $F|_{X'}=\tilde f|_{X'}=f$. Furthermore, for every point 
$p_k\in P$ the following identities hold in the ring $\cO_{X,p_k}/\mgot_{X,p_k}^{n_k}$ of $(n_k-1)$-jets at $p_k$:
\[
  	F_{p_k} = \tilde f_{p_k} - h_{p_k} = \tilde f_{p_k} - \bigl(\tilde f_{p_k} - g_k\bigr)
  	=	g_k = f_k \quad \mod \mgot_{X,p_k}^{n_k}.  
\]
Thus $F$ satisfies the conclusion of the lemma.
\end{proof}

\begin{remark}
\label{rem2:th1}
Lemma \ref{lem:interpolate} gives a version of Theorem \ref{th:th1} in which the function $f$ is assumed to be defined 
and holomorphic only on the subvariety $X'\subset X$ and on a neighborhood of a compact $\cO(X)$-convex set $K\subset X$.
Furthermore, we are given germs $f_k\in \cO_{X,p_k}$ at points $p_k\in P$ such that the conditions of 
Lemma \ref{lem:interpolate} hold. Then for any choice
of integers $n_k\in \N$ there exists a function $F\in\cO(X)$ satisfying Theorem \ref{th:th1}, 
except that conditions (i) and (ii) are replaced by the following conditions:
\begin{itemize}
\item[\rm (i')]  $F|_{X'}=f$, and 
\item[\rm (ii')] $F-f_k$ vanishes to order $n_k$ at each point $p_k\in P$. 
\end{itemize}
\end{remark}

%%%%%%%%%%%%%%%%%%%%%%%%%%%%%%%%%%%%%%%%%%%%%%%%%%%%%%%
%%%%%%%%%%%%%%%%%%%%%%%%%%%%%%%%%%%%%%%%%%%%%%%%%%%%%%%
%																				%	
%																				%
%  Stratified noncritical functions on Stein spaces                      							    	         %
%																				%	
%																				%
%%%%%%%%%%%%%%%%%%%%%%%%%%%%%%%%%%%%%%%%%%%%%%%%%%%%%%%
%%%%%%%%%%%%%%%%%%%%%%%%%%%%%%%%%%%%%%%%%%%%%%%%%%%%%%%

\section{Stratified noncritical functions on Stein spaces}  \label{sec:stratified}

In this section we prove Theorem \ref{th:stratified} on the existence of stratified noncritical holomorphic functions.
As shown in the Introduction, this will also prove Theorems \ref{th:main} and \ref{th:mainbis}.

\begin{proof}[Proof of Theorem \ref{th:stratified}]
Let $(X,\Sigma)$ be a stratified Stein space (see \S\ref{sec:intro}). For every integer $i\in \Z_+$ we let $\Sigma_i$ denote 
the collection of all strata of dimension at most $i$ in $\Sigma$, and let $X_i$ denote the union of all strata in the family 
$\Sigma_i$ (the $i$-skeleton of $\Sigma$). Since the boundary of any stratum is a union of lower dimensional strata, 
$X_i$ is a closed complex subvariety of $X$ of dimension $\le i$ for every $i\in\Z_+$. Clearly $\dim X_i=i$ precisely 
when $\Sigma$ contains at least one $i$-dimensional stratum; otherwise $X_i=X_{i-1}$.  We have 
$X_0\subset X_1\subset \cdots\subset \bigcup_{i=0}^\infty X_i=X$, the sequence $X_i$ is stationary on any 
compact subset of $X$, and $(X_i,\Sigma_i)$ is a stratified Stein subspace of $(X,\Sigma)$ for every $i$. 
Note that $X_0=\{p_1,p_2,\ldots\}$ is a discrete subset of $X$. 

By the assumption of Theorem \ref{th:stratified} we are given for each $p_k\in X_0$ a germ $f_k\in\cO_{X,p_k}$. 
Our task is to find a $\Sigma$-noncritical function $F \in\cO(X)$ which agrees with the germ $f_k$ at $p_k\in X_0$ to 
order $n_k\in \N$.  If the germs $f_k \in \cO_{X,p_k}$ $(p_k\in X_0$) are chosen (strongly) noncritical and $n_k\ge 2$ 
for each $k$,  then the resulting function $F\in\cO(X)$ will be (strongly) noncritical on $X$.
% in the sense of Definition \ref{def:critical}.

Let $F_0\colon X_0\to\C$ be the function on the zero dimensional skeleton defined by $F_0(p_k)=f_k(p_k)$ for every 
$p_k\in X_0$. We shall inductively construct a sequence of functions $F_i\in\cO(X_i)$ satisfying the following conditions 
for $i=1,2,\ldots$:
\begin{itemize}
\item[\rm (i)]     $F_i|_{X_{i-1}} = F_{i-1}$, 
\item[\rm (ii)]    $F_i-f_k|_{X_i}$ vanishes to order $n_k$ at every point $p_k \in X_0$, and
\item[\rm (iii)]   $F_i$ is a stratified noncritical function on the stratified Stein space $(X_i,\Sigma_i)$. \end{itemize}
Assuming that we have found functions $F_1,\ldots,F_{i-1}$ with these properties, we now explain how to find the next function 
$F_i$ in the sequence. If $X_i=X_{i-1}$ then we can simply take $F_i=F_{i-1}$. If this is not the case, then 
$X_i\setminus X_{i-1}$ is a complex manifold of dimension $i$. Apply Lemma \ref{lem:interpolate} with 
$X=X_i$, $X'=X_{i-1}$ and $f=F_{i-1}$ to find a function $G_i\in \cO(X)$ (called $F$ in the lemma) which satisfies 
\begin{itemize}
\item[\rm (i')]    $G_i|_{X_{i-1}} = F_{i-1}$, and
\item[\rm (ii')]   $(G_i)_{p_k} - (f_k|_{X_i})_{p_k} \in \mgot_{X_i,p_k}^{n_k}$ at every point $p_k \in X_0$. 
\end{itemize}
Now $G_i$ satisfies the hypotheses of Theorem \ref{th:th1} (with $X=X_i$, $X'=X_{i-1}$ and $f=G_i$), so we get a function 
$F_i\in\cO(X)$ which agrees with $G_i$ on $X_{i-1}$, it agrees with $G_i$ (and hence with $f_k$) to order $n_k$ at 
$p_k\in X_0$ for each $k$, and is noncritical on $X_i\setminus X_{i-1}$. Hence $F_i$ satisfies properties (i)--(iii) and 
the induction may proceed. 

Since the sequence of subvarieties $X_0\subset X_1\subset \cdots\subset \bigcup_{i=0}^\infty X_i=X$ is stationary on any
compact subset of $X$, the sequence of functions $F_i\in\cO(X_i)$ obtained in this way determines a holomorphic function 
$F\in\cO(X)$ by setting $F=F_i$ on $X_i$ for any $i\in\N$. % (no convergence process is needed). 
It is immediate that $F$ satisfies the conclusion of Theorem \ref{th:stratified}. 
\end{proof}

A similar construction yields the following result.

\begin{theorem} \label{th:stratbis} 
Given a reduced Stein space $X$, a closed complex subvariety $X'$ of  $X$ and a function $f \in \cO(X')$, there exists 
$F\in \cO(X)$ such that $F|_{X'}=f$ and $F$ is strongly noncritical on $X\setminus X'$. Alternatively, we can choose 
$F$ to have critical points at a prescribed discrete set $P$ in $X$ which is contained in $X\setminus X'$. 
\end{theorem}

\begin{proof}
We can stratify the difference $X\setminus X'=\bigcup_j S_j$ into a union of pairwise disjoint connected complex manifolds 
(strata) such that
\begin{itemize}
\item the boundary $bS_j=\overline S_j \setminus S_j$ of any stratum is contained in the union of $X'$ and  of lower 
dimensional strata, 
\item every point of $P$ is a zero dimensional stratum, and 
\item every compact set in $X$ intersects at most finitely many strata.
\end{itemize}
Consider the increasing chain 
$X' \subset X_0 \subset X_1\subset\cdots \subset \bigcup_{i=1}^\infty X_i =X$
of closed complex subvarieties,  where $X_i$ is the union of $X'$ and all strata $S_j$ of dimension at most $i$. 
In particular, we have $P\cup X' \subset X_0$. Then $X_i\setminus X_{i-1}$ is either empty or a disjoint union of 
$i$-dimensional complex manifolds contained in $X\setminus X'$. 

Let $P=\{p_1,p_2,\ldots\}$, and assume that we are given germs $f_k\in \cO_{X,p_k}$ and integeres $n_k\in\N$. 
We start  with the function $F_0\in \cO(X_0)$ which agrees with $f$ on $X'$ and satisfies $F(p_k)=f_k(p_k)$ for every $p_k \in P$. 
By Lemma \ref{lem:interpolate} we can find a function $G_1\in \cO(X)$ which agrees with $F_0$ on $X_0$ and satisfies 
$(G_1)_{p_k} - f_k\in \mgot_{X,p_k}^{n_k}$ at each point $p_k\in P$. 
Theorem \ref{th:th1}, applied to the Stein pair $X_0\subset X_1$ 
and the function $G_1$ (denoted $f$ in the theorem), furnishes a function $F_1\in\cO(X_1)$ which agrees with $G_1$ 
(and hence with $F_0$) on $X_0$ and satisfies $(F_1)_{p_k} - (f_k|_{X_1})_{p_k} \in \mgot_{X_1,p_k}^{n_k}$ 
for every $p_k\in P$. This completes the first step of the induction. By using again Lemma \ref{lem:interpolate} and then 
Theorem \ref{th:th1} we find the next function $F_2\in\cO(X_2)$ such that $F_2|_{X_1}=F_1$ and 
$(F_2)_{p_k} - (f_k|_{X_2})_{p_k} \in  \mgot_{X_2,p_k}^{n_k}$ at every point $p_k\in P$. 
Clearly this process can be continued inductively. We obtain a sequence $F_i\in\cO(X_i)$ for $i=1,2,\ldots$ 
such that the function $F\in\cO(X)$, defined by $F|_{X_i}=F_i$ for every $i=1,2,\ldots$, satisfies the conclusion of 
Theorem \ref{th:stratbis}. 
\end{proof}

\begin{corollary}
\label{cor:extension}
Let $X$ be a reduced Stein space, $X'$ a closed complex subvariety of $X$ without isolated points, and $f\in\cO(X')$ a 
noncritical holomorphic function. Then there exists a noncritical function $F\in\cO(X)$ such that $F|_{X'}=f$.
\end{corollary}

\medskip
\textit{Acknowledgements.}
Research on this paper was supported in part by the program P1-0291 and the grant J1-5432 from ARRS, Republic of Slovenia.

I wish to thank the referee for the remarks which led to improved presentation.

\bibliographystyle{amsplain}

\begin{thebibliography}{10}

\bibitem{Abraham}
{\scshape Abraham, R.:}
Transversality in manifolds of mappings.
Bull.\ Amer.\ Math.\ Soc., \textbf{69}, 470--474 (1963)

\bibitem{AlF1} 
{\scshape Alarc\'on, A.; Forstneri\v c, F.:}
Every bordered Riemann surface is a complete proper curve in a ball. 
Math.\ Ann. \textbf{357} (2013) 1049--1070 

\bibitem{AlF3} 
{\scshape Alarc\'on, A.; Forstneri\v c, F.:}
The Calabi-Yau problem, null curves, and Bryant surfaces.

\texttt{arxiv.org/abs/1308.0903}

\bibitem{AHV} 
{\scshape Aroca, J.M.; Hironaka, H.; Vicente, J.L.:}  
Desingularization theorems.
Mem.\ Math.\ Inst.\ Jorge Juan, no.\ 30, Madrid (1977)

\bibitem{BM} 
{\scshape Bierstone, E.; Milman, P.D.:}
Canonical desingularization in characteristic zero by blowingup
the maximum strata of a local invariant.
Invent.\ Math. \textbf{128} (1997) 207--302 

\bibitem{DGZ}
{\scshape Deng, F.; Guan, Q.; Zhang, L.:} 
Some properties of squeezing functions on bounded domains. 
Pacific J. Math. \textbf{257} (2012)  319?-341

\bibitem{DFW}
{\scshape Diederich, K.; Fornaess, J.E.; Wold, E.F.:}
Exposing points on the boundary of a strictly pseudoconvex or a locally convexifiable domain of 
finite $1$-type. 
 J.\ Geom.\ Anal. \textbf{24} (2014) 2124–-2134 

\bibitem{Docquier-Grauert}
{\scshape Docquier, F.; Grauert, H.:}
Levisches Problem und Rungescher Satz f\"ur Teilgebiete Steinscher Mannigfaltigkeiten.
Math.\ Ann. \textbf{140} (1960) 94--123 


\bibitem{BDF}
{\scshape  Drinovec Drnov\v sek, B.; Forstneri\v c, F.:}
Holomorphic curves in complex spaces.
Duke Math.\ J., \textbf{139}  (2007) 203--254

\bibitem{Fischer}
{\scshape  Fischer, G.:}
Complex Analytic Geometry.
Lecture Notes in Math., vol.\ 538, Springer-Verlag, Berlin (1976)

\bibitem{FW2014}
Forn\ae ss, J.E.; Wold, E.F.:
An Estimate for the Squeezing function and estimates of invariant metrics
\texttt{http://arxiv.org/abs/1411.3846}

\bibitem{Forster1970}
{\scshape  Forster, O.:}
Plongements des vari\'et\'es de Stein.
Comment.\ Math.\ Helv.\ \textbf{45} (1970)  170--184

\bibitem{FF:CI}
{\scshape Forstneri\v c, F.:} 
On complete intersections. 
Ann.\ Inst.\ Fourier \textbf{51} (2001) 497--512

\bibitem{FF:Acta}
{\scshape Forstneri\v c, F.:} 
Noncritical holomorphic functions on Stein manifolds.
Acta Math. \textbf{191} (2003) 143--189

\bibitem{FF:submersions}
{\scshape Forstneri\v c, F.:}
Holomorphic submersions from Stein manifolds.
Ann.\ Inst.\ Fourier\ \textbf{54} (2004)  1913--1942

\bibitem{FF:book}
{\scshape Forstneri\v c, F.:} 
Stein Manifolds and Holomorphic Mappings (The Homotopy Principle in Complex Analysis). 
Ergebnisse der Mathematik und ihrer Grenzgebiete,  
3.\ Folge, 56.  Springer-Verlag, Berlin-Heidelberg (2011)

\bibitem{FO} 
{\scshape Forstneri\v c, F.; Ohsawa, T.:}
Gunning-Narasimhan's theorem with a growth condition
J.\ Geom.\ Anal. \textbf{23} (2013) 1078--1084

% \texttt{http://dx.doi.org/10.1007/s12220-011-9274-0}

\bibitem{FP2}  
{\scshape Forstneri\v c, f.; Prezelj, J.:}
Oka's principle for holomorphic submersions with sprays. 
Math.\ Ann. \textbf{322} (2002) 633--666 

\bibitem{FP3}
{\scshape F.\ Forstneri\v c; J.\ Prezelj:}
Extending holomorphic sections from complex subvarieties. 
Math.\ Z. \textbf{236} (2001) 43--68

\bibitem{FW2009} 
{\scshape Forstneri\v c, F.; Wold, E.F.:}
Bordered Riemann surfaces in $\C^2$. 
J.\ Math.\ Pures Appl. \textbf{91} (2009) 100--114

\bibitem{FW2013} 
{\scshape Forstneri\v c, F.; Wold, E.F.:}
Embeddings of infinitely connected planar domains into $\mathbb C^2$.
Anal.\ PDE \textbf{6} (2013) 499--514

\bibitem{Grauert:modif}
{\scshape Grauert, H.:}
\"Uber Modifikationen und exzeptionelle analytische Mengen. 
Math.\ Ann.\ \textbf{146} (1962) 331--368

\bibitem{Grauert:q-convexity}
{\scshape Grauert, H.:}
Theory of $q$-convexity and $q$-concavity.
In: Several complex variables, VII.
Encyclopaedia Math.\ Sci., vol.\ 74, pp.\ 259--284.
Springer-Verlag, Berlin (1994)

\bibitem{GR-Stellenalgebren}
{\scshape Grauert, H.; Remmert, R.:}
Analytische Stellenalgebren. 
% Unter Mitarbeit von O.\ Riemenschneider. 
Die Grundlehren der mathematischen Wissenschaften, 176. 
Springer-Verlag, Berlin-New York (1971)

\bibitem{Grauert-Remmert1979}
{\scshape Grauert, H.; Remmert, R.:}
Theory of Stein Spaces.
Die Grundlehren der mathematischen Wissenschaften, 227. 
Springer-Verlag, New York (1979)

\bibitem{Gunning-Narasimhan}
{\scshape Gunning, R.C.; Narasimhan, R.:}
Immersion of open Riemann surfaces.
Math.\ Ann. \textbf{174} (1967) 103--108

\bibitem{HL:TF}
{\scshape Henkin, G.M.; Leiterer, J.:}
Theory of Functions on Complex Manifolds.
Akademie-Verlag, Berlin (1984)

\bibitem{Hironaka}
{\scshape H.\ Hironaka:}
Desingularization of complex-analytic varieties, (in French). 
Actes du Congr\`es International des Math\'ematiciens (Nice, 1970), Tome 2, 
627--631, Gauthier-Villars, Paris (1971)

\bibitem{JKS}
{\scshape Ji, S.; Koll\'ar, J.; Shiffman, B.:}
A global \L ojasiewicz inequality for algebraic varieties.
Trans. Amer. Math. Soc. \textbf{329}  (1992)  813--818 

\bibitem{KK}
{\scshape Kaup, L; Kaup, B.:} 
Holomorphic functions of several variables. An introduction to the fundamental theory. 
%With the assistance of Gottfried Barthel. Translated from the German by Michael Bridgland. 
De Gruyter Studies in Mathematics, 3. Walter de Gruyter \& Co., Berlin (1983)

\bibitem{Lieb-Michel}
{\scshape Lieb, I.; Michel, J.:}
The Cauchy-Riemann complex. 
Integral formul\ae\ and Neumann problem.
Aspects of Mathematics, \textbf{E34},
Friedr.\ Vieweg \& Sohn, Braunschweig (2002)

\bibitem{Majcen2009}
{\scshape Majcen, I.:}
Embedding certain infinitely connected subsets 
of bordered Riemann surfaces properly into $\C^2$.  
J.\ Geom.\ Anal. \textbf{19}  (2009) 695--707

\bibitem{Narasimhan}
{\scshape Narasimhan, R.:}
Imbedding of holomorphically complete complex spa\-ces.
Amer.\ J.\ Math. \textbf{82} (1960) 917--934

\bibitem{Osserman}
{\scshape Osserman, R.:}
A survey of minimal surfaces. Second ed. 
Dover Publications, Inc., New York (1986)

\bibitem{Remmert}
{\scshape Remmert, R.:}
Sur les espaces analytiques holomorphiquement s\'eparables et holomorphiquement convexes. 
C.\ R.\ Acad.\ Sci.\ Paris \textbf{243} (1956) 118--121

\bibitem{Remmert2}
{\scshape Remmert, R.:}
Holomorphe und meromorphe Abbildungen komplexer R\"au\-me.
Math.\ Ann., \textbf{133} (1957) 328--370

\bibitem{Richberg}
{\scshape Richberg, R.:}
Stetige streng pseudoconvexe Funktionen.
Math.\ Ann. \textbf{175} (1968) 257--286

\bibitem{Siu1976}
{\scshape Siu, J.-T.:}
Every Stein subvariety admits a Stein neighborhood.
Invent.\ Math. \textbf{38} (1976) 89--100

\bibitem{Stopar} {\scshape Stopar, K.:}
Approximation of holomorphic mappings on $1$-convex domains. 
Internat.\ J.\ Math.  \textbf{24} (2013), no.\ 14, 1350108, 33 pp.

\bibitem{Whitney} 
{\scshape  Whitney, H.:}
Local properties of analytic varieties.  
In: Differentiable and Combinatorial Topology,
A Symposium in honor of Marston Morse, S.\ Ca\-irns, ed., pp.\ 205--244.  
Princeton University Press, Princeton (1965)

\bibitem{Whitney2}
{\scshape  Whitney, H.:} Complex Analytic Varieties.
Addison-Wesley, Reading (1972)

\end{thebibliography}

\vskip 0.5cm

\noindent {\scshape Franc Forstneri\v c}

\noindent Faculty of Mathematics and Physics, University of Ljubljana, and 

\noindent Institute of Mathematics, Physics and Mechanics 

\noindent Jadranska 19, 1000 Ljubljana, Slovenia

\noindent e-mail: {\tt franc.forstneric@fmf.uni-lj.si}

\end{document}